\documentclass[a4paper,10pt,british]{article}

\usepackage[T1]{fontenc}
\usepackage[utf8]{inputenc}

\usepackage{amsmath}
\usepackage{amssymb}

\usepackage{lmodern}

\usepackage{babel}
\usepackage{csquotes}
\usepackage[babel]{microtype}
\usepackage[shortcuts]{extdash}

\usepackage[backref=true,doi=true,maxnames=100,sorting=nyt]{biblatex}
\renewbibmacro{in:}{}

\numberwithin{equation}{subsection}

\usepackage{amsthm}
\theoremstyle{plain}
\newtheorem{assertion}{Assertion}[subsection]

\newtheorem{corollary}[assertion]{Corollary}

\newtheorem{lemma}[assertion]{Lemma}
\newtheorem{proposition}[assertion]{Proposition}
\newtheorem{theorem}[assertion]{Theorem}

\theoremstyle{definition}

\newtheorem{definition}[assertion]{Definition}

\newtheorem{examples}[assertion]{Examples}

\theoremstyle{remark}

\newtheorem{notation}[assertion]{Notation}

\newtheorem{remark}[assertion]{Remark}

\usepackage{tikz-cd}

\usepackage[hidelinks]{hyperref}
\ifluatex
\else

\fi

\newcommand{\Ac}{\mathcal{A}}
\newcommand{\Bc}{\mathcal{B}}
\newcommand{\Cc}{\mathcal{C}}

\newcommand{\Ec}{\mathcal{E}}
\newcommand{\Fc}{\mathcal{F}}

\newcommand{\Hc}{\mathcal{H}}

\newcommand{\Lc}{\mathcal{L}}
\newcommand{\Mc}{\mathcal{M}}

\newcommand{\Oc}{\mathcal{O}}

\newcommand{\Xc}{\mathcal{X}}

\newcommand{\Zc}{\mathcal{Z}}

\newcommand{\N}{\mathbf{N}}
\newcommand{\Z}{\mathbf{Z}}
\newcommand{\Q}{\mathbf{Q}}

\newcommand{\C}{\mathbf{C}}

\newcommand{\Qbar}{\overline{\Q}}

\newcommand{\Zp}{\Z_p}
\newcommand{\Qp}{\Q_p}
\newcommand{\Cp}{\C_p}

\newcommand{\Qpbar}{\Qbar_p}

\newcommand{\et}{\mathrm{ét}}
\newcommand{\dR}{\mathrm{dR}}

\newcommand{\cris}{\mathrm{cris}}
\newcommand{\e}{\mathrm{e}}
\newcommand{\Bb}{\mathbf{B}}
\newcommand{\BdR}{\Bb_\dR}

\newcommand{\Bcris}{\Bb_\cris}
\newcommand{\Be}{\Bb_\e}

\newcommand{\dprime}{{\prime\prime}}

\newcommand{\Mod}{\mathrm{Mod}}
\newcommand{\Rep}{\mathrm{Rep}}
\newcommand{\Bun}{\mathrm{Bun}}
\newcommand{\Coh}{\mathrm{Coh}}
\newcommand{\XFF}{X^\mathrm{FF}}
\newcommand{\ab}{\mathrm{ab}}
\newcommand{\ur}{\mathrm{ur}}
\newcommand{\free}{\mathrm{free}}
\newcommand{\fl}{\mathrm{fl.}}
\newcommand{\gfl}{\mathrm{g.fl.}}
\newcommand{\tor}{\mathrm{tor}}
\newcommand{\HN}{\mathrm{HN}}
\newcommand{\HT}{\mathrm{HT}}
\newcommand{\Iw}{\mathrm{Iw}}
\newcommand{\pcris}{\mathrm{pcris}}
\newcommand{\disc}{\mathrm{disc}}

\DeclareMathOperator{\Hom}{Hom}

\DeclareMathOperator{\Ker}{Ker}
\DeclareMathOperator{\Img}{Im}

\DeclareMathOperator{\Ext}{Ext}

\DeclareMathOperator{\Spec}{Spec}
\DeclareMathOperator{\Proj}{Proj}

\DeclareMathOperator{\DD}{D}
\DeclareMathOperator{\HH}{H}
\DeclareMathOperator{\KK}{K}
\DeclareMathOperator{\Fil}{Fil}

\DeclareMathOperator{\Gal}{Gal}

\DeclareMathOperator{\Sym}{Sym}
\DeclareMathOperator{\height}{ht}
\DeclareMathOperator{\rank}{rank}
\DeclareMathOperator{\length}{length}
\DeclareMathOperator{\res}{res}
\DeclareMathOperator{\cores}{cores}
\DeclareMathOperator{\id}{id}
\DeclareMathOperator{\DDb}{\mathbf{D}}

\newcommand{\similarrightarrow}{\xrightarrow{\sim}}

\newcommand{\keywords}[1]{\noindent{\small\textbf{Keywords:}~#1.}\par}
\newcommand{\msc}[2]{\noindent{\small\textbf{Mathematics Subject Classification (2020):}~Primary:~#1. Secondary:~#2.}\par}
\newcommand{\email}[1]{\href{mailto:#1}{\nolinkurl{#1}}}

\title{\textbf{Bloch--Kato groups over perfectoid fields and Galois theory of \(p\)\=/adic periods}}
\author{Gautier \textsc{Ponsinet}}
\date{\today}

\begin{document}
\maketitle
\begin{abstract}
	We relate the structure of the Bloch\--Kato groups associated with a de Rham Galois representation over a perfectoid field to the Galois theory of the ring \(\BdR^+\) of \(p\)\=/adic periods.
	As an application, we answer the question raised by Coates and Greenberg and motivated by Iwasawa theory to compute the Bloch\--Kato groups over perfectoid fields in new cases, generalising results of Coates and Greenberg and the author.
	Our method relies on the classification of vector bundles over the Fargues\--Fontaine curve.
\end{abstract}
\bigskip
\msc{\texttt{11R23}}{\texttt{14G45}}
\tableofcontents

\section{Introduction}
\label{sec:introduction}

The present article is concerned with the question raised by Coates and Greenberg~\cite[131]{CoatesGreenberg1996} and motivated by Iwasawa theory to compute the Bloch\--Kato groups associated with a de Rham Galois representation over a perfectoid field.
The purpose of the present article is to answer Coates and Greenberg's question in new cases, generalising the results and the method of the article~\cite{Ponsinet2024:1}.

\subsection{Motivation}
\label{subsec:motivation}

Let \(p\) be a prime number.
Let \(\Qpbar\) be an algebraic closure of the field \(\Qp\) of \(p\)\=/adic numbers.
Let \(K\) be a finite extension of \(\Qp\) contained in \(\Qpbar\).
We denote by \(G_K = \Gal(\Qpbar/K)\) the absolute Galois group of \(K\).
Let \(V\) be a \(p\)\=/adic representation of \(G_K\), that is, a finite dimensional \(\Qp\)\=/vector space equipped with a continuous and \(\Qp\)\=/linear action of \(G_K\), and let \(T\) be a \(\Zp\)\=/lattice in \(V\) stable under the action of \(G_K\), that is, a finitely generated \(\Zp\)\=/submodule of \(V\), spanning \(V\) over \(\Qp\), and stable under the action of \(G_K\).

Bloch and Kato~\cite{BlochKato1990} have defined, for each finite extension \(K^\prime\) of \(K\), subgroups in Galois cohomology using \(p\)\=/adic Hodge theory (see \S \ref{subsec:BK})
\[
	\HH^1_e(K^\prime,V/T) \subseteq \HH^1_f(K^\prime,V/T) \subseteq \HH^1_g(K^\prime,V/T) \subseteq \HH^1(K^\prime,V/T),
\]
which are involved in their conjecture on the special values of \(L\)\=/functions of motives~\cite{Fontaine1992}.

Let \(L\) be an algebraic extension of \(K\) contained in \(\Qpbar\).
For each \(\ast \in \{e,f,g\}\), we consider the group
\[
	\HH^1_\ast(L,V/T) = \varinjlim_{\res, K^\prime} \HH^1_\ast(K^\prime,V/T),
\]
where \(K^\prime\) runs over all the finite extensions of \(K\) contained in \(L\), and the transition morphisms are the restriction maps.
The Bloch\--Kato subgroups thus defined
\[
	\HH^1_e(L,V/T) \subseteq \HH^1_f(L,V/T) \subseteq \HH^1_g(L,V/T) \subseteq \HH^1(L,V/T)
\]
are involved in the Iwasawa main conjecture for motives~\cite{CoatesFukayaKatoSujathaVenjakob2005,FukayaKato2006}.

Coates and Greenberg~\cite[131]{CoatesGreenberg1996} have raised the question motivated by Iwasawa theory to compute the Bloch\--Kato groups when the completion \(\hat{L}\) of \(L\) for the \(p\)\=/adic valuation topology is a perfectoid field~\cite[\S 3]{Scholze2012}.
In particular, if \(\hat{L}\) is a perfectoid field, then Coates and Greenberg have computed the Bloch\--Kato groups when \(T\) is the \(p\)\=/adic Tate module associated with an abelian variety \(A\) defined over \(K\) (see Remark~\ref{intro:remark:previous_results}), in which case~\cite[Example~3.11]{BlochKato1990} \(V/T\) is the group of \(p\)\=/power torsion points \(A[p^\infty]\) of \(A\), and the Bloch\--Kato groups are all equal and coincide with the image of the Kummer map
\[
	A(L) \otimes_{\Z} \Qp/\Zp \rightarrow \HH^1(L,A[p^\infty]).
\]

We refer the reader to the introduction of the articles~\cite{CoatesGreenberg1996,Perrin-Riou2000,Ponsinet2024:1} and the references therein for more details about this question, its history and its motivation from Iwasawa theory which can be traced back to the foundational article~\cite{Mazur1972} by Mazur.

\begin{remark}
	Coates and Greenberg~\cite{CoatesGreenberg1996} use the notion, which they have introduced, of deeply ramified extensions.
	Recall that the extension \(L/K\) is deeply ramified if and only if the field \(\hat{L}\) is perfectoid (see~\cite[Remark~3.3]{Scholze2012} or~\cite[Lemma~2.21]{KedlayaPottharst2018}).
\end{remark}

\begin{remark}
	Let \(T^\ast(1) = \Hom_{\Zp}(T,\Zp(1))\) be the Tate dual representation of \(T\).
	Let
	\[
		\HH^1_\Iw(K,L,T^\ast(1)) = \varprojlim_{\cores, K^\prime} \HH^1(K^\prime,T^\ast(1))
	\]
	be the first Iwasawa cohomology group of the extension \(L/K\) with coefficients in \(T^\ast(1)\) , where \(K^\prime\) runs over all the finite extensions of \(K\) contained in \(L\), and the transition morphisms are the corestriction maps.
	Bloch and Kato have also defined subgroups
	\[
		\HH^1_e(K^\prime,T^\ast(1)) \subseteq \HH^1_f(K^\prime,T^\ast(1)) \subseteq \HH^1_g(K^\prime,T^\ast(1)) \subseteq \HH^1(K^\prime,T^\ast(1)).
	\]
	For each \(\ast \in \{e,f,g\}\), the modules of \(\ast\)\=/universal norms associated with \(T^\ast(1)\) in the extension \(L/K\) is defined by
	\[
		\HH^1_{\Iw,\ast}(K,L,T^\ast(1)) = \varprojlim_{\cores, K^\prime} \HH^1_\ast(K^\prime,T^\ast(1)).
	\]
	Local Tate duality induces a perfect pairing
	\[
		\HH^1(L,V/T) \times \HH^1_\Iw(K,L,T^\ast(1)) \rightarrow \Qp/\Zp.
	\]
	If \(V\) is de Rham, then under local Tate duality (see \S \ref{subsec:universal_norms}) the groups
	\begin{equation} \label{intro:eq:BK}
		\HH^1_e(L,V/T) \subseteq \HH^1_f(L,V/T) \subseteq \HH^1_g(L,V/T)
	\end{equation}
	are respectively the orthogonal complements of the modules of universal norms
	\begin{equation} \label{intro:eq:universal_norms}
		\HH^1_{\Iw,g}(K,L,T^\ast(1)) \supseteq \HH^1_{\Iw,f}(K,L,T^\ast(1)) \supseteq \HH^1_{\Iw,e}(K,L,T^\ast(1)).
	\end{equation}
	To compute the Bloch\--Kato groups~\eqref{intro:eq:BK} is therefore equivalent to compute the modules of universal norms~\eqref{intro:eq:universal_norms}.
\end{remark}

\subsection{Main results}
\label{subsec:results}

We will first prove the following relation between the Bloch\--Kato groups.
We consider the group \(\HH^1(L,V/T)\) and its subgroups as discrete \(\Zp\)\=/modules.
Recall that the Pontryagin dual of a discrete \(\Zp\)\=/module \(M\) is the compact Hausdorff \(\Zp\)\=/module \(\Hom_{\Zp}(M,\Qp/\Zp)\).

\begin{proposition} \label{intro:proposition:comparison}
	If \(V\) is de Rham, then the Pontryagin dual of the quotient
	\[
		\HH^1_g(L,V/T)/\HH^1_e(L,V/T)
	\]
	is a free \(\Zp\)\=/module of finite rank bounded independently of \(L\).
\end{proposition}

From the perspective of Iwasawa theory, the Bloch\--Kato groups are thus closely related with each other, and it is therefore enough to study one of them.
We study the exponential Bloch\--Kato group \(\HH^1_e(L,V/T)\).

We need the following notation to state the main result of the present article.

\begin{itemize}
	\item The definition of the Bloch\--Kato groups (see \S \ref{subsec:BK}) involves the field of \(p\)\=/adic periods \(\BdR\) and subrings of \(\BdR\) introduced by Fontaine~\cite{Fontaine1994:II}.
	In particular, the natural filtration \((\Fil^n \BdR)_{n \in \Z}\) on \(\BdR\) induces a decreasing separated and exhaustive filtration by subgroups
	\[
		(\Fil^n \HH^1_e(L,V/T))_{n \in \Z}
	\]
	on the exponential Bloch\--Kato groups \(\HH^1_e(L,V/T)\).
	\item Let \(\BdR^+ = \Fil^0 \BdR\) be the ring of \(p\)\=/adic periods, and for each integer \(n \geq 1\), let \(\Bb_n = \BdR^+/\Fil^n \BdR\).
	Recall that \(\BdR^+\) is endowed with a canonical topology and a continuous action by \(G_K\) which induces a structure on each \(\Bb_n\) of \(\Qp\)\=/Banach space equipped with a continuous and \(\Qp\)\=/linear action by \(G_K\).
	Recall also that \(\Qpbar\) can be identified with a subfield of \(\BdR^+\), and we have \(\Qpbar \subset \Bb_n\).
	In particular, there is an inclusion \(L \subset \Bb_n^{G_L}\), and we consider \(\Bb_n^{G_L}\) endowed with the subspace topology from \(\Bb_n\).
	\item If \(V\) is de Rham, then, for each integer \(n \geq 1\), we denote by \(V^{\leq 0,> n}\) the maximal quotient representation of \(V\) whose Hodge\--Tate weights are all in the set \(\Z \setminus [1,n]\), and by \(T^{\leq 0,> n}\) the image of \(T\) in \(V^{\leq 0,> n}\).
	The quotient map \(V/T \rightarrow V^{\leq 0,> n}/T^{\leq 0,> n}\) induces a morphism
	\[
		\pi_{0,n} : \HH^1(L,V/T) \rightarrow \HH^1(L,V^{\leq 0,> n}/T^{\leq 0,> n}).
	\]
	Note that if the Hodge\--Tate weights of \(V\) are all \(\leq n\), then the representation \(V^{\leq 0,> n}\) is the maximal quotient representation of \(V\) whose Hodge\--Tate weights are all \(\leq 0\), and we then simply denote the representation \(V^{\leq 0,> n}\) by \(V^{\leq 0}\), the lattice \(T^{\leq 0,> n}\) by \(T^{\leq 0}\), and the map \(\pi_{0,n}\) by
	\[
		\pi_0 : \HH^1(L,V/T) \rightarrow \HH^1(L,V^{\leq 0}/T^{\leq 0}).
	\]
\end{itemize}

The main result of the present article is then the following.

\begin{theorem} \label{intro:theorem:main}
	Let \(n \geq 1\) be an integer.
	If \(V\) is de Rham and if \(\hat{L}\) is a perfectoid field such that \(L\) is dense in \(\Bb_n^{G_L}\), then the map \(\pi_{0,n}\) induces an isomorphism
	\[
		\HH^1(L,V/T)/\Fil^{-n} \HH^1_e(L,V/T) \similarrightarrow \HH^1(L,V^{\leq 0,> n}/T^{\leq 0,> n}).
	\]
\end{theorem}

From Theorem~\ref{intro:theorem:main}, we then obtain the following descriptions of the full exponential Bloch\--Kato group.

\begin{corollary} \label{intro:corollary:main}
	Let \(n \geq 1\) be an integer.
	Assume that \(V\) is de Rham and that \(\hat{L}\) is a perfectoid field such that \(L\) is dense in \(\Bb_n^{G_L}\).
	\begin{enumerate}
		\item \label{intro:corollary:main:1} If the quotient representation \(V^{\leq 0,> n}\) is trivial, then
		\[
			\HH^1_e(L,V/T) = \HH^1(L,V/T).
		\]
		\item \label{intro:corollary:main:2} If the Hodge\--Tate weights of \(V\) are all \(\leq n\), then the map \(\pi_0\) induces an isomorphism
		\[
			\HH^1(L,V/T)/\HH^1_e(L,V/T) \similarrightarrow \HH^1(L,V^{\leq 0}/T^{\leq 0}).
		\]
	\end{enumerate}
\end{corollary}

Theorem~\ref{intro:theorem:main} therefore relates the structure of the Bloch\--Kato groups over perfectoid fields to the Galois theory of the ring \(\BdR^+\) of \(p\)\=/adic periods.

\begin{examples} \label{intro:examples:density}
	We recall results about the density of \(L\) in \(\Bb_n^{G_L}\).
	\begin{enumerate}
		\item Colmez~\cite{Colmez1994,Colmez2012} has proved that \(\Qpbar\) is dense in \(\BdR^+\).
		\item If \(\hat{L}\) is \emph{not} a perfectoid field, then Iovita and Zaharescu~\cite[Theorem~0.1]{IovitaZaharescu1999:1} have proved that, for each \(n \in \N\), there are isomorphisms of topological groups
		\[
			(\Bb_\dR^+)^{G_L} \similarrightarrow \Bb_n^{G_L} \similarrightarrow \hat{L}.
		\]
		Hence, if the field \(\hat{L}\) \emph{not} perfectoid, then \(L\) is dense in \((\Bb_\dR^+)^{G_L}\).
		\item Let \(\Cp\) be the completion of \(\Qpbar\) for the \(p\)\=/adic valuation topology.
		Recall that there exists an isomorphism \(\Bb_1 \similarrightarrow \Cp\) of \(p\)\=/adic Banach representations of \(G_K\).
		By the Ax\--Sen\--Tate theorem~\cite{Tate1967}, there are isomorphisms of \(p\)\=/adic Banach spaces
		\[
			\Bb_1^{G_L} \similarrightarrow \Cp^{G_L} \similarrightarrow \hat{L}.
		\]
		Hence, the field \(L\) is always dense in \(\Bb_1^{G_L}\).
		\item If \(L/K\) is an infinitely ramified \(\Zp\)\=/extension, then the field \(\hat{L}\) is perfectoid, and Berger~\cite{Berger2024} has proved that \(L\) is \emph{not} dense in \(\Bb_2^{G_L}\).
		The case of the cyclotomic \(\Zp\)\=/extension has also been proved by Colmez~\cite[Proposition~8.2]{IovitaZaharescu1999:1}.
		\item If \(L= K(p^{1/p^{\infty}})\) is the extension generated over \(K\) by all the \(p\)\=/power roots of \(p\), then the field \(\hat{L}\) is perfectoid, and Iovita and Zaharescu~\cite[Corollary~8.1]{IovitaZaharescu1999:1} have proved that \(L\) is \emph{not} dense in \(\Bb_2^{G_L}\).
		\item Let \(m \geq 2\). If \(K=\Q_{p^m}\) is the unique unramified extension of \(\Qp\) with residue field with \(p^m\) elements and if \(L = K^{\ab}\) is the maximal abelian extension of \(K\), then the field \(\hat{L}\) is perfectoid, and Iovita and Zaharescu~\cite[Corollary~8.2]{IovitaZaharescu1999:1} have proved that \(L\) is dense in \(\Bb_2^{G_L}\).
	\end{enumerate}
	Note that Iovita and Zaharescu end their article~\cite{IovitaZaharescu1999:1} with open problems concerning the Galois theory of \(\BdR^+\).
\end{examples}

\begin{remark} \label{intro:remark:previous_results}
	By the Ax\--Sen\--Tate theorem, Theorem~\ref{intro:theorem:main} and Corollary~\ref{intro:corollary:main} holds for any perfectoid field \(\hat{L}\) in the case \(n=1\).
	In particular, the point~\ref{intro:corollary:main:2} of the Corollary~\ref{intro:corollary:main} in the case \(n=1\) is the main result of the article~\cite{Ponsinet2024:1} and a generalisation of the aforementioned theorem by Coates and Greenberg~\cite{CoatesGreenberg1996} for abelian varieties. Indeed, recall~\cite{Tate1967,Fontaine1982} that the rational \(p\)\=/adic Tate module \(V_p(A)\) associated with an abelian variety \(A\) defined over \(K\) is a de Rham \(p\)\=/adic representation of \(G_K\) whose Hodge\--Tate weights are all in \([0,1]\), and thus, the point~\ref{intro:corollary:main:2} of the Corollary~\ref{intro:corollary:main} applies to \(V_p(A)\).
	The new computations of the exponential Bloch\--Kato groups obtained in Theorem~\ref{intro:theorem:main} and Corollary~\ref{intro:corollary:main} thereby answer the question raised by Coates and Greenberg~\cite[131]{CoatesGreenberg1996} in new cases.
\end{remark}

\begin{remark}
	If \(L\) is the cyclotomic \(\Zp\)\=/extension of \(K\), then \(L\) is not dense in \(\Bb_2^{G_L}\) by the aforementioned results of Berger and Colmez, hence Theorem~\ref{intro:theorem:main} and Corollary~\ref{intro:corollary:main} apply only in the case \(n = 1\).
	However, Berger~\cite{Berger2005}, generalising results by Perrin-Riou~\cite{Perrin-Riou2000,Perrin-Riou2001}, has computed the Bloch\--Kato groups associated with a de Rham Galois representation over the cyclotomic extension without any restriction on the Hodge\--Tate weights of the representation.
	Berger has proved that if \(V\) is de Rham and if \(L\) is the cyclotomic \(\Zp\)\=/extension of \(K\), then the map \(\pi_0\) induces a surjective morphism
	\[
		\HH^1(L,V/T)/\HH^1_e(L,V/T) \rightarrow \HH^1(L,V^{\leq 0}/T^{\leq 0}) \rightarrow 0,
	\]
	of which the Pontryagin dual of its kernel is a finitely generated \(\Zp\)\=/module.
\end{remark}

\begin{remark}
	Additionally to Coates and Greenberg's open question~\cite[131]{CoatesGreenberg1996} to compute the Bloch\--Kato groups over perfectoid fields, Büyükboduk~\cite[Conjectures~2.5, 2.6, and~2.7]{Büyükboduk2015} has conjectured that the structure of the Bloch\--Kato groups over the anticyclotomic \(\Zp\)\=/extension should be similar to the structure of the Bloch\--Kato groups over the cyclotomic \(\Zp\)\=/extension computed by Berger~\cite{Berger2005} .
\end{remark}

\begin{remark}
	We precise the role of perfectoid fields in Theorem~\ref{intro:theorem:main} and Corollary~\ref{intro:corollary:main}.
	Let \(G\) be a connected \(p\)\=/divisible group of height \(\height(G)\) and dimension \(\dim(G)\) defined over the ring of integers \(\Oc_K\) of \(K\).
	Let \(T_p(G)\) be the \(p\)\=/adic Tate module of \(G\), and let \(V_p(G) = \Qp \otimes_{\Zp} T_p(G)\) be the rational \(p\)\=/adic Tate module of \(G\), hence, \(V_p(G)/T_p(G) = G[p^\infty]\) is the group of \(p\)\=/power torsion points of \(G\).
	Recall~\cite[Example~3.10]{BlochKato1990} that the exponential Bloch\--Kato group \(\HH^1_e(L,G[p^\infty])\) then coincides with the image of the Kummer map
	\[
		G(\Oc_L) \otimes_{\Zp} \Qp/\Zp \rightarrow \HH^1(L,G[p^\infty]).
	\]
	Recall~\cite{Tate1967,Fontaine1982} that the representation \(V_p(G)\) is de Rham with Hodge\--Tate weights in \([0,1]\) and that the quotient representation \(V_p(G)^{\leq 0}\) is the rational \(p\)\=/adic Tate module \(V_p(G^{\et})\) associated with the maximal étale quotient \(G^\et\) of \(G\) which is therefore trivial since \(G\) is assumed to be connected.
	The point~\ref{intro:corollary:main:2} of Corollary~\ref{intro:corollary:main} in the case \(n=1\) then applies to \(G\) and thus the following statement holds.
	\begin{itemize}
		\item If \(\hat{L}\) is a perfectoid field, then we have \(\HH^1_e(L,G[p^\infty]) = \HH^1(L,G[p^\infty])\).
	\end{itemize}

	On the one hand, if \(\height(G) = \dim(G)\), Coates and Greenberg~\cite[Proposition~4.7]{CoatesGreenberg1996} have proved that if \(L/K\)  is an infinitely wildly ramified, then \(\HH^1_e(L,G[p^\infty]) = \HH^1(L,G[p^\infty])\).
	(Coates and Greenberg state this result for abelian varieties, their proof is valid for \(p\)\=/divisible groups.)
	Note that if \(\hat{L}\) is perfectoid, then \(L/K\) is infinitely wildly ramified~\cite[Lemma~2.12]{CoatesGreenberg1996}.
	Thus, for specific representations, there exists a wider class of fields than perfectoid fields for which the equality \(\HH^1_e(L,V/T) = \HH^1(L,V/T)\) holds.

	On the other hand, if \(\height(G) > \dim(G)\), then Bondarko~\cite{Bondarko2003}, generalising a result by Coates and Greenberg~\cite[Proposition~5.4]{CoatesGreenberg1996}, has proved that the following statements are equivalent.
	\begin{enumerate}
		\item The field \(\hat{L}\) is perfectoid.
		\item The valuation of \(L\) is non-discrete and \(\HH^1_e(L,G[p^\infty]) = \HH^1(L,G[p^\infty])\).
	\end{enumerate}
\end{remark}

\begin{remark}
	Perfectoid fields are ubiquitous in Iwasawa theory.
	Indeed, the fields presented in Examples~\ref{intro:examples:density} are perfectoid.
	Moreover, recall that if the extension \(L/K\) is Galois with Galois group a \(p\)\=/adic Lie group in which the inertia subgroup is infinite, then the field \(\hat{L}\) is perfectoid by Sen (see~\cite{Sen1972} and~\cite[Theorem~2.13]{CoatesGreenberg1996}).
	All such fields are studied in Iwasawa theory~\cite{CoatesFukayaKatoSujathaVenjakob2005,FukayaKato2006}.
\end{remark}

\begin{remark}
	We briefly mention applications of our results in Iwasawa theory.
	Results such as the main results of the present article, Theorem~\ref{intro:theorem:main} and Corollary~\ref{intro:corollary:main}, and Berger's result~\cite{Berger2005} for the cyclotomic extension allow to precisely compare the Bloch\--Kato Selmer groups to Selmer groups à la Greenberg~\cite{Greenberg1989}.
	The Bloch\--Kato Selmer groups are involved in the Iwasawa main conjecture while Selmer groups à la Greenberg are more accessible to study.
\end{remark}

\subsection{Outline of the proof}
\label{subsec:outline}

The method of the present article to study the Bloch\--Kato groups over perfectoid fields generalises and improves on the approach developed in~\cite{Ponsinet2024:1}.
It relies on:
\begin{itemize}
	\item the theory of almost \(\Cp\)\=/representations introduced by Fontaine~\cite{Fontaine2003},
	\item the classification of coherent sheaves over the Fargues\--Fontaine curve due to Fargues and Fontaine~\cite{FarguesFontaine2018}, and
	\item the relation between almost \(\Cp\)\=/representations and coherent sheaves over the Fargues\--Fontaine curve established by Fontaine~\cite{Fontaine2020}.
\end{itemize}

The category \(\Cc(G_K)\) of almost \(\Cp\)\=/representations of \(G_K\) is an abelian full subcategory of the category of \(p\)\=/adic Banach representations of \(G_K\) introduced by Fontaine~\cite{Fontaine2003} which contains as full subcategories the category of \(p\)\=/adic representations of \(G_K\) and the category of torsion \(\BdR^+\)\=/representations of \(G_K\).

Let \(n \geq 1\) be an integer.
Fontaine has associated with \(V\) a short exact sequence of almost \(\Cp\)\=/representations of \(G_K\)
\begin{equation} \label{intro:eq:E+n}
		0 \rightarrow V \rightarrow E_+^n(V) \rightarrow \Fil^{-n} \hat{t}_V(\Qpbar) \rightarrow 0,
\end{equation}
such that
\begin{itemize}
	\item \(\Fil^{-n} \hat{t}_V(\Qpbar)\) is a trivial \(\Bb_n\)\=/representation of \(G_K\), that is, there exists an isomorphism in \(\Cc(G_K)\)
	\[
		\Fil^{-n} \hat{t}_V(\Qpbar) \similarrightarrow  \bigoplus_{i \in [1,n]} \Bb_i^{\oplus m_i(V)},
	\]
	for some integers \(m_i(V) \in \N\).
	Moreover, if \(V\) is de Rham, then \(m_i(V)\) is the multiplicity of \(i\) as a Hodge\--Tate weight of \(V\).
\item \(E_+^n(V)\) is the universal extension of \(V\) by a trivial \(\Bb_n\)\=/representation in \(\Cc(G_K)\).
\end{itemize}

Let \(E_\delta^n(V/T)\) be the maximal discrete \(G_K\)\=/submodule of \(E_+^n(V)/T\).
Then, we will prove that there is an isomorphism
\begin{equation} \label{intro:eq:BK_delta}
	\HH^1(L,V/T)/\Fil^{-n} \HH^1_e(L,V/T) \similarrightarrow \HH^1(L,E_\delta^n(V/T)).
\end{equation}
Moreover, if \(\hat{L}\) is a perfectoid field such that \(L\) is dense in \(\Bb_n^{G_L}\), then we will prove that the computation of \(\HH^1(L,E_\delta^n(V/T))\) reduces to the computation of \(\HH^1(L,E_+^n(V))\) which we will then carry out as follows.

Let \(\XFF\) be the Fargues\--Fontaine curve~\cite{FarguesFontaine2018}, and let \(\Coh_{\XFF}(G_K)\) be the category of \(G_K\)\=/equivariant coherent sheaves over \(\XFF\).
Fontaine~\cite{Fontaine2020} has proved that the global section functor on \(\XFF\) induces an equivalence of triangulated categories between the bounded derived categories
\begin{equation} \label{intro:eq:equiv_derived}
	\DD^b(\Coh_{\XFF}(G_K)) \similarrightarrow \DD^b(\Cc(G_K)).
\end{equation}
Moreover, if \(\hat{L}\) is a perfectoid field, then Fargues and Fontaine have classified the \(G_L\)\=/equivariant coherent sheaves over \(\XFF\).
We will deduce from these results the following.

\begin{proposition} \label{intro:cohomology_almost_perfectoid}
	Let \(X\) be an almost \(\Cp\)\=/representation of \(G_K\).
	Let \(X^0\) be the maximal quotient \(p\)\=/adic representation of \(X\).
	If \(\hat{L}\) is a perfectoid field, then the quotient map \(X \rightarrow X^0\) induces an isomorphism
	\[
		\HH^1(L,X) \similarrightarrow \HH^1(L,X^0).
	\]
\end{proposition}

Therefore, in order to compute \(\HH^1(L,E_+^n(V))\), it remains to identify the maximal quotient \(p\)\=/adic representation \(E_+^n(V)^0\) of \(E_+^n(V)\) which we will do by considering \(E_+^n(V)\) in terms of vector bundles over the Fargues\--Fontaine curve.

By the equivalence~\eqref{intro:eq:equiv_derived}, the short exact sequence~\eqref{intro:eq:E+n} is isomorphic to the global sections of a short exact sequence of \(G_K\)\=/equivariant coherent sheaves over \(\XFF\)
\[
	0 \rightarrow \Ec(V) \xrightarrow{\eta} \Ec_+^n(V) \rightarrow \Fc_+^n(V) \rightarrow 0,
\]
where \(\Ec(V) = \Oc_{\XFF} \otimes_{\Qp} V\) is the vector bundle associated with \(V\) over \(\XFF\).
If \(V\) is de Rham, then \(\Ec(V)\) and \(\Ec_+^n(V)\) are de Rham vector bundles over \(\XFF\), and we will prove that \(\eta: \Ec(V) \rightarrow \Ec_+^n(V)\) is then solution of the following universal problem.

Let \(\Bun_{\XFF}(G_K)_\dR\) be the category of de Rham \(G_K\)\=/equivariant vector bundles over \(\XFF\), and let \(\Bun_{\XFF}(G_K)_\dR^{\leq 0,>n}\) be the full subcategory of \(\Bun_{\XFF}(G_K)_\dR\) of de Rham \(G_K\)\=/equivariant vector bundles over \(\XFF\) whose Hodge\--Tate weights are all in the set \(\Z \setminus [1,n]\).
The forgetful functor from \(\Bun_{\XFF}(G_K)_\dR^{\leq 0,>n}\) to \(\Bun_{\XFF}(G_K)_\dR\) admits a left adjoint
\[
	\begin{split}
		\tau_\dR^{\leq 0,>n} : \Bun_{\XFF}(G_K)_\dR & \rightarrow \Bun_{\XFF}(G_K)_\dR^{\leq 0,>n} \\
		\Ec & \mapsto \Ec_+^n.
	\end{split}
\]
If \(V\) is de Rham, then \(\Ec_+^n(V) = \tau_\dR^{\leq 0,>n}(\Ec(V))\) and \(\eta: \Ec(V) \rightarrow \Ec_+^n(V)\) is the universal morphism from \(\Ec(V)\) to the forgetful functor from \(\Bun_{\XFF}(G_K)_\dR^{\leq 0,>n}\) to \(\Bun_{\XFF}(G_K)_\dR\).
We will then easily deduce from the universal property of \(\Ec_+^n(V)\) that the maximal quotient \(p\)\=/adic representation \(E_+^n(V)^0\) of \(E_+^n(V)\) is the representation \(V^{\leq 0,> n}\).

\subsection{Organisation of the article}
\label{subsec:organisation}

In section~\ref{sec:almost}, we review the definition and some properties of almost \(\Cp\)\=/representations due to Fontaine~\cite{Fontaine2003,Fontaine2020}.
We also briefly review the \(p\)\=/adic period rings defined by Fontaine~\cite{Fontaine1994:II}.

In section~\ref{sec:XFF}, we review the properties of coherent sheaves over the Fargues\--Fontaine curve~\cite{FarguesFontaine2018} which will be needed.
We recall the relation~\eqref{intro:eq:equiv_derived} between almost \(\Cp\)\=/representations and coherent sheaves over the Fargues\--Fontaine curve established by Fontaine~\cite{Fontaine2020}.
We study the maximal quotient \(p\)\=/adic representation of an almost \(\Cp\)\=/representation.

In section~\ref{sec:perfectoid}, we establish Proposition~\ref{intro:cohomology_almost_perfectoid}.
Then, given an almost \(\Cp\)\=/representation \(X\), an almost \(\Cp\)\=/subrepresentation \(Z\) of \(X\), and a Galois stable lattice \(\Zc\) in \(Z\), we study the cohomology of the maximal discrete Galois submodule \((X/\Zc)_\delta\) of \(X/\Zc\), which will allow us to relate the cohomology of \(E_\delta^n(V/T)\) to the cohomology of \(E_+^n(V)\).

In section~\ref{sec:truncation}, we define and study the functor \(\tau_\dR^{\leq 0,>n}\).
We then study the vector bundle \(\Ec_+^n(V)\).

In section~\ref{sec:BK}, we review the definition of the Bloch\--Kato groups and we define the filtration on the exponential Bloch\--Kato group.
We then review the definition and properties of the almost \(\Cp\)\=/representation \(E_+^n(V)\), and of the discrete Galois module \(E_\delta^n(V/T)\) associated with \(V/T\) by Fontaine.
We then establish the relation~\eqref{intro:eq:BK_delta} between the cohomology of \(E_\delta^n(V/T)\) and the filtered part of the exponential Bloch\--Kato group, and the relation between \(E_+^n(V)\) and \(\Ec_+^n(V)\).
Finally, we prove the main results stated in the introduction: Proposition~\ref{intro:proposition:comparison}, Theorem~\ref{intro:theorem:main}, and Corollary~\ref{intro:corollary:main}.

\subsection{Acknowledgements}
\label{subsec:acknowledgements}

This work has been supported by:
\begin{itemize}
	\item the grant \enquote{assegni di collaborazione ad attività di ricerca \enquote{Ing. Giorgio Schirillo}} from the Istituto Nazionale di Alta Matematica \enquote{Francesco Severi} which has allowed me to conduct research at the Università degli studi di Genova,
	\item the grant ANR-18-CE40-0029 \enquote{Représentations galoisiennes, formes automorphes et leurs fonctions \(L\)} which has allowed me to conduct research at the Université de Bordeaux,
	\item the Deutsche Forschungsgemeinschaft (DFG, German Research Foundation) through the Collaborative Research Centre TRR 326 \enquote{Geometry and Arithmetic of Uniformized Structures}, project number 444845124, which has allowed me to conduct research at the Universität Heidelberg, and
	\item the Institut des Hautes Études Scientifiques.
\end{itemize}
I thank Stefano Vigni, Denis Benois, and Otmar Venjakob at these respective institutions.

I also thank Laurent Berger for discussions about his article~\cite{Berger2024} and the present article.

\subsection{Notation}
\label{subsec:notation}

We adopt the convention that the set of natural numbers \(\N\) contains \(0\).

We fix a prime number \(p\), and an algebraic closure \(\Qpbar\) of the field \(\Qp\) of \(p\)\=/adic numbers.
We denote by \(\Cp\) the completion of \(\Qpbar\) for the \(p\)\=/adic valuation topology.
Every algebraic extension of \(\Qp\) considered is contained in \(\Qpbar\).
We denote by \(\Qp^\ur\) the maximal unramified extension of \(\Qp\).
If \(L\) is an algebraic extension of \(\Qp\), then we denote by \(G_L = \Gal(\Qpbar/L)\) the absolute Galois group of \(L\), by \(L_0 = L \cap \Qp^\ur\) the maximal unramified extension of \(\Qp\) contained in \(L\), and by \(\hat{L}\) the completion of \(L\) for the \(p\)\=/adic valuation topology.
We also fix a finite extension \(K\) of \(\Qp\).

If \(G\) is a topological group, then a topological \(G\)\=/module is a topological abelian group \(M\) equipped with a continuous action of \(G\) compatible with the group structure of \(M\), and a discrete \(G\)\=/module is a topological \(G\)\=/module whose underlying topological space is discrete.
If \(G\) is a topological group, and if \(M\) is a topological \(G\)\=/module, then, for each \(n \in \N\), we denote by \(\HH^n(G,M)\) the \(n\)\=/th group of continuous group cohomology of \(G\) with coefficients in \(M\) (see~\cite[\S 2]{Tate1976} or~\cite[Appendix~A]{Ponsinet2024:1}).
Recall that if
\begin{equation} \label{eq:ses}
	0 \rightarrow M^\prime \rightarrow M \rightarrow M^\dprime \rightarrow 0
\end{equation}
is a short exact sequence of topological \(G\)\=/modules such that the topology of \(M^\prime\) is the subspace topology from \(M\), and the topology of \(M^\dprime\) is the quotient topology from \(M\), then the short exact sequence~\eqref{eq:ses} induces an exact sequence
\begin{equation} \label{eq:esh0h1}
	\begin{tikzcd}
		0 \ar{r} & \HH^0(G,M^\prime) \ar{r} & \HH^0(G,M) \ar{r} \ar[phantom, ""{coordinate, name=Z}]{d} & \HH^0(G,M^\dprime) \ar[rounded corners, to path={ -- ([xshift=2em]\tikztostart.east) |- (Z) [near end]\tikztonodes -| ([xshift=-2em]\tikztotarget.west) -- (\tikztotarget)}]{dll} \\
		& \HH^1(G,M^\prime) \ar{r} & \HH^1(G,M) \ar{r} & \HH^1(G,M^\dprime).
	\end{tikzcd}
\end{equation}
Moreover, if there exists a continuous section of the projection of \(M\) on \(M^\dprime\) as topological space, then the exact sequence~\eqref{eq:esh0h1} extends into a long exact sequence
\[
	\cdots \rightarrow \HH^n(G,M^\prime) \rightarrow \HH^n(G,M) \rightarrow  \HH^n(G,M^\dprime) \rightarrow \HH^{n+1}(G,M^\prime) \rightarrow \cdots.
\]
In particular, such a continuous section exists if \(M^\dprime\) is a discrete \(G\)\=/module.

If \(G_k\) is the absolute Galois group of a field \(k\), and if \(M\) is a topological \(G_k\)\=/module, then, for each \(n \in \N\), we write \(\HH^n(k,M)\) instead of \(\HH^n(G_k,M)\), and \(\HH^n(k,M)\) is the \(n\)\=/th group of Galois cohomology of \(k\) with coefficients in \(M\).

We denote by \(\Zp(1)\) the free \(\Zp\)\=/module of rank \(1\) whose elements are sequences \((\zeta_{p^n})_{n \in \N}\) of \(p\)\=/power roots of unity in \(\Qpbar\) such that \(\zeta_1 = 1\) and \(\zeta_{p^{n+1}}^p = \zeta_{p^n}\) for each \(n \in \N\), endowed with the natural action of \(G_{\Qp}\) by the cyclotomic character \(\chi\).
We fix a generator \(t\) of \(\Zp(1)\) with group law written additively.
For each \(n \in \N\), we set
\[
	\begin{split}
		\Zp(n)  & = \Sym^n_{\Zp}(\Zp(1)), \\
		\Zp(-n) & = \Hom_{\Zp}(\Zp(n),\Zp).
	\end{split}
\]
Then, for each \(n \in \Z\), the Galois group \(G_{\Qp}\) acts on \(\Zp(n)\) by \(\chi^{n}\).

If \(L\) is an algebraic extension of \(\Qp\) and if \(M\) is a \(\Zp\)\=/module equipped with a linear action of \(G_L\), then, for each \(n \in \Z\), then the \emph{\(n\)\=/th Tate twist} \(M(n)\) of \(M\) is the \(\Zp\)\=/module
\[
	M(n) = M \otimes_{\Zp} \Zp(n)
\]
on which \(G_L\) acts by \(g(m \otimes z) = g(m) \otimes g(z)= \chi^n(g) (g(m)\otimes z)\), for all \(g \in G_L\), \(m \in M\) and \(z \in \Zp(n)\).

\section{Almost \(\Cp\)-representations}
\label{sec:almost}

\subsection{\(p\)-adic Banach representations and almost \(\Cp\)-representations}
\label{subsec:almost}

We recall the definition of the category of almost \(\Cp\)\=/representations of \(G_K\) introduced by Fontaine~\cite{Fontaine2003}.

A \emph{\(p\)\=/adic Banach space} is a topological \(\Qp\)\=/vector space whose topology is the one of a Banach space over \(\Qp\).
Thus, we consider Banach spaces over \(\Qp\) up to norm equivalence.

A \emph{\(p\)\=/adic Banach representation of \(G_K\)} is a \(p\)\=/adic Banach space equipped with a continuous and \(\Qp\)\=/linear action of \(G_K\).
A morphism of \(p\)\=/adic Banach representations of \(G_K\) is a \(G_K\)\=/equivariant continuous and \(\Qp\)\=/linear map.
We denote by \(\Bc(G_K)\) the category of \(p\)\=/adic Banach representations of \(G_K\).

A \emph{\(G_K\)\=/stable lattice} in a \(p\)\=/adic Banach representation \(X\) of \(G_K\) is an open \(\Zp\)\=/submodule of \(X\) which is complete and separated for the \(p\)\=/adic topology, spans \(X\) over \(\Qp\), and is stable under the action of \(G_K\).

Note that the field \(\Cp\) endowed with its natural topology and its natural action of \(G_K\) is a \(p\)\=/adic Banach representation of \(G_K\).
A \emph{trivial \(\Cp\)\=/representation of \(G_K\)} is a \(p\)\=/adic Banach representation of \(G_K\) isomorphic to \(\Cp^{\oplus d}\) for some \(d \in \N\).

If \(X\) and \(Y\) are two \(p\)\=/adic Banach representations of \(G_K\), then an \emph{almost isomorphism} from \(X\) to \(Y\) is a triple \((V,U,\alpha)\) composed of:
\begin{enumerate}
	\item a finite dimensional \(G_K\)\=/stable \(\Qp\)\=/vector subspace \(V\) of \(X\),
	\item a finite dimensional \(G_K\)\=/stable \(\Qp\)\=/vector subspace \(U\) of \(Y\), and
	\item an isomorphism in \(\Bc(G_K)\)
	\[
		\alpha: X/V \similarrightarrow Y/U.
	\]
\end{enumerate}
Almost isomorphisms form an equivalence relation on \(p\)\=/adic Banach representations of \(G_K\).
Two \(p\)\=/adic Banach representations of \(G_K\) are almost isomorphic if there exists an almost isomorphism between them.

An \emph{almost \(\Cp\)\=/representation of \(G_K\)} is a \(p\)\=/adic Banach representation of \(G_K\) which is almost isomorphic to a trivial \(\Cp\)\=/representation of \(G_K\).
We denote by \(\Cc(G_K)\) the full subcategory of \(\Bc(G_K)\) of almost \(\Cp\)\=/representations of \(G_K\).

\begin{theorem}[Fontaine] \label{theorem:almost_abelian}
	The category \(\Cc(G_K)\) is abelian.
\end{theorem}

Let \(\Rep_{\Qp}(G_K)\) be the category of \emph{\(p\)\=/adic representations of \(G_K\)}, that is, the category of finite dimensional \(\Qp\)\=/vector spaces equipped with a continuous and \(\Qp\)\=/linear action of \(G_K\).
The category \(\Rep_{\Qp}(G_K)\) defines a Serre subcategory of \(\Cc(G_K)\) which we denote by \(\Cc^0(G_K)\).

\subsection{\(p\)-adic period rings}
\label{subsec:period}

We briefly review the \(p\)\=/adic period rings defined by Fontaine~\cite{Fontaine1994:II}.

The \emph{ring of \(p\)\=/adic periods} \(\BdR^+\) is a complete discrete valuation ring endowed with an action of \(G_{\Qp}\), whose residue field is \(\Cp\), and of which \(t\) is a uniformiser.

The \emph{field of \(p\)\=/adic periods} \(\BdR\) is the field of fractions of \(\BdR^+\).
There is a natural filtration on \(\BdR\) by the fractional ideals
\[
	\Fil^n \BdR = \BdR^+ \cdot t^n, n \in \Z,
\]
which is stable under the action of \(G_{\Qp}\).
For each \(n \in \N\), we set
\[
	\Bb_n = \BdR^+ / \Fil^n \BdR.
\]
In particular, there is an isomorphism \(\Bb_1 \similarrightarrow \Cp\).

The field \(\BdR\) is equipped with a topology, the so-called \emph{canonical} topology (see~\cite[\S 1.5.3]{Fontaine1994:II} or~\cite{Colmez2012}), which is coarser than the valuation topology from \(\BdR^+\).
The action of \(G_{\Qp}\) on \(\BdR\) endowed with the canonical topology is continuous, and, we have
\begin{equation} \label{eq:FilGalois}
	\begin{aligned}
		(\Fil^n \BdR)^{G_K} & = \BdR^{G_K} = K & \text{if } n \leq 0, \\
		(\Fil^n \BdR)^{G_K} & = 0	       & \text{if } n > 0.
	\end{aligned}
\end{equation}
Unless otherwise stated, we will consider \(\BdR\) and its subquotient rings endowed with the canonical topology.
In particular, for each integer \(n >0\), \(\Bb_n\) is a \(p\)\=/adic Banach representation of \(G_{\Qp}\), and the isomorphism \(\Bb_1 \similarrightarrow \Cp\) is an isomorphism of \(p\)\=/adic Banach representations of \(G_{\Qp}\).

Moreover, the map \(\BdR^+ \rightarrow \Cp\) induces an isomorphism between the separable closure of \(\Qp\) in \(\BdR^+\) and the field of the \(p\)\=/adic algebraic numbers \(\Qpbar\), which we use to identify these two fields.
The field \(\Qpbar\) is then dense in \(\BdR^+\) (\cite{Colmez1994,Colmez2012}).

The ring \(\BdR^+\) contains a subring \(\Bcris^+\), stable under the action of \(G_{\Qp}\), containing \(t\), and equipped with an endomorphism \(\varphi\) commuting with the action of \(G_{\Qp}\).
The \emph{ring of crystalline periods} is \(\Bcris=\Bcris^+[1/t]\).
We have
\[
	\Bcris^{G_K} = K_0.
\]
Moreover, we have \(\varphi(t) = p t\), and the endomorphism \(\varphi\) extends uniquely to \(\Bcris\).

Let
\[
	\Be = \Bcris^{\varphi = 1} = \{ b \in \Bcris, \varphi(b)=b \},
\]
and, for each \(n \in \N\), let
\[
	(\Bcris^+)^{\varphi = p^n} = \{ b \in \Bcris^+, \varphi(b)=p^n b \}.
\]
The ring \(\Be\) is a principal ideal domain~\cite[Théorème~6.5.2]{FarguesFontaine2018}.
Moreover, the ring \(\Be\) inherits a filtration from \(\BdR\), and we have
\[
	\Fil^n \Be = \Fil^n \BdR \cap \Be =
	\begin{cases}
		(\Bcris^+)^{\varphi = p^{-n}} \cdot t^n, & \text{if } n \leq 0, \\
		0,					& \text{if } n > 0.
	\end{cases}
\]
Furthermore, we have
\[
	\Fil^0 \Be = (\Bcris^+)^{\varphi = 1} = \Qp,
\]
and there exists \(G_{\Qp}\)\=/equivariant short exact sequences of topological \(\Qp\)\=/algebras, the so-called \emph{fundamental exact sequences},
\begin{equation} \label{eq:funda}
	0 \rightarrow \Qp \rightarrow \Be \rightarrow \BdR/\BdR^+ \rightarrow 0,
\end{equation}
 and, for each \(n \in \N\),
\begin{equation} \label{eq:funda_fil}
	0 \rightarrow \Qp \rightarrow \Fil^{-n} \Be \rightarrow \Fil^{-n} \BdR/\BdR^+ \rightarrow 0.
\end{equation}

\subsection{Almost \(\Cp\)-representations and \(p\)-adic periods}
\label{subsec:almost_periods}

We recall the relation between \(\BdR^+\)\=/representations of \(G_K\) and almost \(\Cp\)\=/representations of \(G_K\) established by Fontaine~\cite{Fontaine2003}.

Let \(\Rep_{\BdR^+}^\tor(G_K)\) be the category of finitely generated torsion \(\BdR^+\)\=/modules equipped with a continuous and \(\BdR^+\)\=/semilinear action of \(G_K\).
We call an object of \(\Rep_{\BdR^+}^\tor(G_K)\) a \emph{torsion \(\BdR^+\)\=/representation of \(G_K\)}.

Note that the \(\BdR^+\)\=/module underlying a torsion \(\BdR^+\)\=/representation of \(G_K\) is a \(\Bb_n\)\=/module for some sufficiently large integer \(n\).
Since \(\Bb_n\) is a \(\Qp\)\=/Banach space, a torsion \(\BdR^+\)\=/representation of \(G_K\) is a \(p\)\=/adic Banach representation of \(G_K\).
Hence, there is a forgetful functor from \(\Rep_{\BdR^+}^\tor(G_K)\) to \(\Bc(G_K)\) whose essential image we denote by \(\Cc^{+\infty}(G_K)\).

\begin{theorem}[Fontaine]
	The forgetful functor from \(\Rep_{\BdR^+}^\tor(G_K)\) to \(\Bc(G_K)\) is fully faithful.
	Moreover, its essential image \(\Cc^{+\infty}(G_K)\) is a weak Serre subcategory of \(\Cc(G_K)\).
\end{theorem}

\subsection{Torsion pairs}
\label{subsec:torsion_pairs}

We recall the definition and properties of torsion pairs on a abelian category~\cite[\S 1.12]{Borceux1994}.

\begin{definition} \label{definition:torsion_pair}
	Let \(\Ac\) be an abelian category.
	A \emph{torsion pair} on \(\Ac\) is a tuple \((\Bc,\Cc)\) of strictly full subcategories of \(\Ac\) such that
	\begin{enumerate}
		\item for each object \(B\) of \(\Bc\) and each object \(C\) of \(\Cc\), we have
		\[
			\Hom_{\Ac}(B,C) = 0,
		\]
		\item for each object \(A\) of \(\Ac\), there exists a short exact sequence
		\[
			0 \rightarrow B \rightarrow A \rightarrow C \rightarrow 0,
		\]
		with \(B\) an object of \(\Bc\), and \(C\) an object of \(\Cc\).
	\end{enumerate}
\end{definition}

If \((\Bc,\Cc)\) is a torsion pair on an abelian category \(\Ac\), then the definition of a torsion pair implies that for each object \(A\) of \(\Ac\), there exists a unique, up to isomorphism, short exact sequence
\[
	0 \rightarrow B \rightarrow A \rightarrow C \rightarrow 0,
\]
with \(B\) an object of \(\Bc\), and \(C\) an object of \(\Cc\) (see~\cite[Proposition~1.12.4]{Borceux1994}).

\subsection{Effective and coeffective almost \(\Cp\)-representations}
\label{subsec:effective}

We recall the definition of effective and coeffective almost \(\Cp\)\=/representations of \(G_K\) due to Fontaine~\cite[\S 3L, \S 6C]{Fontaine2020}.

An almost \(\Cp\)\=/representation \(X\) of \(G_K\) is \emph{effective} if there exists an object \(Y\) of \(\Cc^{+\infty}(G_K)\) such that \(X\) is isomorphic to a subobject of \(Y\).
We denote by \(\Cc^{\geq 0}(G_K)\) the full subcategory of \(\Cc(G_K)\) of effective almost \(\Cp\)\=/representations of \(G_K\).

An almost \(\Cp\)\=/representation \(X\) of \(G_K\) is \emph{coeffective} if for all object \(Y\) of \(\Cc^{+\infty}(G_K)\), we have
\[
	\Hom_{\Cc(G_K)}(X,Y) = 0.
\]
We denote by \(\Cc^{< 0}(G_K)\) the full subcategory of \(\Cc(G_K)\) of coeffective almost \(\Cp\)\=/representations of \(G_K\).

\begin{proposition}[Fontaine] \label{proposition:effective_coeffective}
	The subcategories \(\Cc^{< 0}(G_K)\) and \(\Cc^{\geq 0}(G_K)\) of \(\Cc(G_K)\) satisfy the following properties.
	\begin{enumerate}
		\item The categories \(\Cc^{< 0}(G_K)\) and \(\Cc^{\geq 0}(G_K)\) are exact subcategories of \(\Cc(G_K)\).
		\item The category \(\Cc^{\geq 0}(G_K)\) is the smallest strictly full subcategory of \(\Cc(G_K)\) containing \(\Cc^0(G_K)\) and \(\Cc^{+\infty}(G_K)\) and stable under taking extensions and direct summands.
		\item The tuple \((\Cc^{< 0}(G_K),\Cc^{\geq 0}(G_K))\) is a torsion pair on \(\Cc(G_K)\).
	\end{enumerate}
\end{proposition}

\section{Coherent sheaves over the Fargues--Fontaine curve}
\label{sec:XFF}

\subsection{The Fargues--Fontaine curve}
\label{subsec:XFF}

We review properties of coherent sheaves over the Fargues\--Fontaine curve~\cite{FarguesFontaine2018}.

Let
\[
	\XFF = \Proj\left(\bigoplus_{n \in \N} (\Bcris^+)^{\varphi = p^n} \right)
\]
be the Fargues\--Fontaine curve.
Recall that the scheme \(\XFF\) is regular, Noetherian, separated, connected, and one-dimensional.
Moreover, the curve \(\XFF\) is complete.

We recall the description à la Beauville\--Laszlo of coherent sheaves over \(\XFF\).

Let \(L\) be an algebraic extension of \(\Qp\).
For a ring \(R \in \{\Be,\BdR^+,\BdR\}\), we denote by \(\Rep_R(G_L)\) the category of finitely generated \(R\)\=/modules equipped with a continuous and \(R\)\=/semilinear action of \(G_L\), and an object of \(\Rep_R(G_L)\) is called a \(R\)\=/representation of \(G_L\).

Let \(\Mc(G_L)\) be the category whose objects are triple \((M_\e,M_\dR^+,\iota_M)\)  composed of:
\begin{enumerate}
	\item a \(\Be\)\=/representation \(M_\e\) of \(G_L\),
	\item a \(\BdR^+\)\=/representation \(M_\dR^+\) of \(G_L\), and
	\item a \(G_L\)\=/equivariant morphism of \(\BdR^+\)\=/modules
	\[
		\iota_M : M_\dR^+ \rightarrow \BdR \otimes_{\Be} M_\e
	\]
	which induces an isomorphism of \(\BdR\)\=/representations of \(G_L\)
	\[
		\BdR \otimes_{\BdR^+} M_\dR^+ \similarrightarrow \BdR \otimes_{\Be} M_\e,
	\]
\end{enumerate}
and whose maps are tuple \(f = (f_\e,f_\dR^+)\) composed of:
\begin{enumerate}
	\item a morphism \(f_\e\) of \(\Be\)\=/representations of \(G_L\), and
	\item a morphism \(f_\dR^+\) of \(\BdR^+\)\=/representations of \(G_L\),
\end{enumerate}
such that the diagram
\[
	\begin{tikzcd}
		M_\dR^+ \ar{d}{f_\dR^+} \ar{r}{\iota_M} & \BdR \otimes_{\Be} M_\e \ar{d}{1 \otimes f_\e} \\
		N_\dR^+ \ar{r}{\iota_N} & \BdR \otimes_{\Be} N_\e
	\end{tikzcd}
\]
commutes.

The Galois group \(G_{\Qp}\) acts on the Fargues\--Fontaine curve \(\XFF\) via its action on the period rings defining \(\XFF\).
We denote by \(\Coh_{\XFF}(G_L)\) the category of \(G_L\)\=/equivariant coherent sheaves over \(\XFF\).

Recall that there exists a unique closed point of \(\XFF\), denoted by \(\infty\), fixed by the action of \(G_{\Qp}\) on \(\XFF\).
The completion of the stalk \(\Oc_{\XFF,\infty}\) is isomorphic to \(\BdR^+\), and \(\XFF \setminus \{\infty\} = \Spec(\Be)\).
Therefore, there is a functor
\begin{equation} \label{eq:functor_BL}
	\Coh_{\XFF}(G_L) \rightarrow \Mc(G_L)
\end{equation}
which associates with a coherent sheaf \(\Fc\) the triple:
\begin{itemize}
	\item \(\Fc_\e = \HH^0(\XFF \setminus \{\infty\}, \Fc)\),
	\item \(\Fc_\dR^+\) the completion of the stalk \(\Fc_\infty\), and
	\item \(\iota_{\Fc}\) the glueing data.
\end{itemize}

\begin{proposition}[Fargues\--Fontaine] \label{theorem:BL}
	The functor~\eqref{eq:functor_BL} induces an equivalence of categories
	\[
		\Coh_{\XFF}(G_L) \similarrightarrow \Mc(G_L).
	\]
\end{proposition}

We denote by \(\Bun_{\XFF}(G_L)\) the full subcategory of \(\Coh_{\XFF}(G_L)\) of \(G_L\)\=/equivariant vector bundles over \(\XFF\).

Under the equivalence of Theorem~\ref{theorem:BL}, the category \(\Bun_{\XFF}(G_L)\) is equivalent to the full subcategory of \(\Mc(G_L)\) whose objects are triple \((M_\e,M_\dR^+,\iota_M)\) such that
\begin{itemize}
	\item the \(\Be\)\=/module underlying \(M_\e\) is a free, and
	\item the \(\BdR^+\)\=/module underlying \(M_\dR^+\) is a free.
\end{itemize}

Moreover, if \(L\) is a finite extension of \(\Qp\), then Fontaine~\cite[Proposition~3.1]{Fontaine2020} has proved that the \(\Be\)\=/module underlying any \(\Be\)\=/representation of \(G_L\) is free.
Therefore, we have the following characterisation of \(G_K\)\=/equivariant vector bundles and \(G_K\)\=/equivariant torsion coherent sheaves over \(\XFF\).

\begin{proposition} \label{proposition:char}
	Let \(\Fc\) be \(G_K\)\=/equivariant coherent sheaf over \(\XFF\).
	\begin{enumerate}
		\item The sheaf \(\Fc\) is a vector bundle if and only if the \(\BdR^+\)\=/module underlying \(\Fc_\dR^+\) is free.
		\item The following statements are equivalent.
		\begin{enumerate}
			\item The sheaf \(\Fc\) is torsion.
			\item The \(\Be\)\=/representation \(\Fc_\e\) is trivial.
			\item The \(\BdR^+\)\=/representation \(\Fc_\dR^+\) is torsion.
		\end{enumerate}
	\end{enumerate}
\end{proposition}

\begin{notation}
	We will write \(\Fc = (\Fc_\e,\Fc_\dR^+,\iota_{\Fc})\) a coherent sheaf over \(\XFF\).
	We will omit the map \(\iota_{\Fc}\) when there is no ambiguity.
\end{notation}

\begin{notation}
	If \(M_\dR^+\) denotes a \(\BdR^+\)\=/module, then we set
	\[
		M_\dR = \BdR \otimes_{\BdR^+} M_\dR^+.
	\]
\end{notation}

\subsection{The Harder--Narasimhan filtration}
\label{subsec:HN_Fil}

We denote by \(\KK_0(\Coh_{\XFF})\) the Grothendieck group associated with the category \(\Coh_{\XFF}\) of coherent sheaves over \(\XFF\).
There exists group homomorphisms \emph{degree} and \emph{rank}
\[
	\begin{split}
		\deg : \KK_0(\Coh_{\XFF})  & \rightarrow \Z, \\
		\rank : \KK_0(\Coh_{\XFF}) & \rightarrow \Z,
	\end{split}
\]
characterised by the following properties.
\begin{itemize}
	\item If \(\Fc\) is a coherent sheaf over \(\XFF\), then \(\rank(\Fc)\) is the rank of \(\Fc\) as an \(\Oc_{\XFF}\)\=/module.
	\item If \(\Lc\) is an invertible sheaf over \(\XFF\), then \(\deg(\Lc)\) is the degree of the divisor associated with \(\Lc\).
	\item If \(\Ec\) is a vector bundle over \(\XFF\) of rank \(r\), and if \(\bigwedge^r \Ec\) is the determinant line bundle associated with \(\Ec\), then
	\[
		\deg(\Ec) = \deg\left(\bigwedge^r \Ec\right).
	\]
	\item If \(\Fc\) is a torsion coherent sheaf over \(\XFF\), then
	\[
		\deg(\Fc) = \sum_{\text{closed point } x \in \XFF} \length_{\Oc_{\XFF,x}}( \Fc_x ).
	\]
\end{itemize}

The \emph{slope} of a coherent sheaf \(\Fc\), denoted by \(\mu(\Fc)\), is the element of \({\Q \cup \{+\infty\}}\) defined by
\[
	\mu(\Fc) =
	\begin{cases}
		+\infty,                       & \text{if } \Fc \text{ is torsion}, \\
		\dfrac{\deg(\Fc)}{\rank(\Fc)}, & \text{otherwise}.
	\end{cases}
\]
A coherent sheaf \(\Fc\) is \emph{semistable} if for each non-zero subsheaf \(\Fc^\prime \subseteq \Fc\), we have \(\mu(\Fc^\prime) \leq \mu(\Fc)\).

If \(\Fc\) is a coherent sheaf over \(\XFF\), there exists a unique filtration of \(\Fc\) by subsheaves
\begin{equation} \label{eq:HN}
	0 = \Fc_0 \subseteq \Fc_1 \subseteq \cdots \subseteq \Fc_{n-1} \subseteq \Fc_n = \Fc
\end{equation}
such that
\begin{enumerate}
	\item the sheaf \(\Fc_i/\Fc_{i-1}\) is semistable for each \(i \in \{1,\ldots,n\}\), and
	\item \(\mu(\Fc_i/\Fc_{i-1}) > \mu(\Fc_{i+1}/\Fc_i)\)  for each \(i \in \{1,\ldots,n-1\}\).
\end{enumerate}
The filtration~\eqref{eq:HN} is the \emph{Harder\--Narasimhan filtration} of \(\Fc\), and the slopes \((\mu(\Fc_i/\Fc_{i-1}))_{i \in \{1,\ldots,n\}}\) are the \emph{Harder\--Narasimhan slopes} of \(\Fc\).

\begin{proposition}[Fargues\--Fontaine] \label{proposition:cohomology_HN}
	Let \(\Fc\) be a coherent sheaf over \(\XFF\).
	For each integer \(n > 1\), the group \(\HH^n(\XFF,\Fc)\) is trivial.
	There is an exact sequence functorial in \(\Fc\)
	\[
		0 \rightarrow \HH^0(\XFF,\Fc) \rightarrow \Fc_\e \oplus \Fc_\dR^+ \xrightarrow{\delta_{\Fc}} \Fc_\dR \rightarrow \HH^1(\XFF,\Fc) \rightarrow 0,
	\]
	where \(\delta_{\Fc}(x,y) = x - \iota_{\Fc}(y)\).
	Moreover,
	\begin{enumerate}
		\item the group \(\HH^0(\XFF,\Fc)\) is trivial if and only if the Harder\--Narasimhan slopes of \(\Fc\) are all \(< 0\), and
		\item the group \(\HH^1(\XFF,\Fc)\) is trivial if and only if the Harder\--Narasimhan slopes of \(\Fc\) are all \(\geq 0\).
	\end{enumerate}
\end{proposition}

Let \(L\) be an algebraic extension of \(\Qp\).

\begin{remark} \label{remark:HN_equiv}
	The uniqueness of the Harder\--Narasimhan filtration implies that the Harder\--Narasimhan filtration of a \(G_L\)\=/equivariant coherent sheaf is composed of \(G_L\)\=/equivariant coherent sheaves.
\end{remark}

If \(\hat{L}\) is a perfectoid field~\cite[\S 3]{Scholze2012}, then Fargues and Fontaine~\cite[Théorème~9.3.1 and Théorème~9.4.1]{FarguesFontaine2018} have classified \(G_L\)\=/equivariant sheaves over \(\XFF\).
Part of the classification is the following.

\begin{theorem}[Fargues\--Fontaine] \label{theorem:FF_classification_perfectoid}
	If \(\hat{L}\) is a perfectoid field, then the Harder--Narasimhan filtration of a \(G_L\)\=/equivariant coherent sheaf over \(\XFF\) is split in \(\Coh_{\XFF}(G_L)\).
\end{theorem}

\subsection{Harder--Narasimhan twists}
\label{subsec:HN_twist}

We recall the definition and properties of the Harder\--Narasimhan twists of coherent sheaves due to Fontaine~\cite[\S 3H]{Fontaine2020}.
If \(\Fc\) is a coherent sheaf over \(\XFF\), then, for each \(n \in \Z\), the \emph{\(n\)\=/th Harder\--Narasimhan twist} of \(\Fc\), denoted by \(\Fc(n)_\HN\), is the coherent sheaf defined as the following modification of \(\Fc\) at the point \(\infty\):
\[
	\Fc(n)_\HN = (\Fc_\e,\Fc_\dR^+(-n),\iota_{\Fc}(-n)).
\]
We then have the following short exact sequences of coherent sheaves
\begin{align}
	0 \rightarrow \Fc \rightarrow \Fc(n)_\HN \rightarrow (0, t^{-n} \Fc_\dR^+/ \Fc_\dR^+) \rightarrow 0 & \quad \text{if } n \geq 0, \label{eq:HNtwistpositive} \\
	0 \rightarrow \Fc(n)_\HN \rightarrow \Fc \rightarrow (0, \Fc_\dR^+/t^{-n} \Fc_\dR^+) \rightarrow 0  & \quad \text{if } n < 0. \label{eq:HNtwistnegative}
\end{align}
Moreover, if \(L\) is an algebraic extension of \(\Qp\) and
if \(\Fc\) is \(G_L\)\=/equivariant, then \(\Fc(n)_\HN\) is \(G_L\)\=/equivariant, and the short exact sequences~\eqref{eq:HNtwistpositive} and~\eqref{eq:HNtwistnegative} are short exact sequences of \(G_L\)\=/equivariant coherent sheaves.

\begin{proposition}[Fontaine] \label{proposition:HNtwist}
	Let \(\Fc\) be a coherent sheaf over \(\XFF\).
	Let \(n \in \Z\).
	We have
	\[
		\mu(\Fc(n)_\HN) = \mu(\Fc) + n.
	\]
	Moreover, if \(\Fc\) is semistable, then \(\Fc(n)_\HN\) is semistable.
\end{proposition}

\begin{corollary}
	Let \(\Fc\) be a coherent sheaf over \(\XFF\), and let \((\Fc_i)_{i \in \{0,\ldots ,n\}}\) be the Harder\--Narasimhan filtration of \(\Fc\).
	Let \(n \in \Z\).
	Then \((\Fc_i(n)_\HN)_{i \in \{0,\ldots ,n\}}\) is the Harder\--Narasimhan filtration of \(\Fc(n)_\HN\).
	In particular, the Harder\--Narasimhan slopes of \(\Fc(n)_\HN\) are \(\{\mu_i + n\}_{i \in \{1,\ldots,n\}}\), where \(\mu_i\) runs over the Harder\--Narasimhan slopes of \(\Fc\).
\end{corollary}

\subsection{Almost \(\Cp\)-representations and coherent sheaves}
\label{subsec:almost_FF}

We recall the relation between almost \(\Cp\)\=/representations of \(G_K\) and \(G_K\)\=/equivariant coherent sheaves over the Fargues\--Fontaine curve established by Fontaine~\cite{Fontaine2020}.

Let \(\Coh_{\XFF}(G_K)\) be the category of \(G_K\)\=/equivariant coherent sheaves over \(\XFF\).
We set the following subcategories of \(\Coh_{\XFF}(G_K)\).
\begin{itemize}
	\item Let \(\Coh_{\XFF}^{\geq 0}(G_K)\) be the full subcategory of \(\Coh_{\XFF}(G_K)\) of \(G_K\)\=/equivariant coherent sheaves whose Harder\--Narasimhan slopes are all \(\geq 0\).
	Note that the category \(\Coh_{\XFF}^{\geq 0}(G_K)\) is an exact subcategory of \(\Coh_{\XFF}(G_K)\).
	\item Let \(\Coh_{\XFF}^{< 0}(G_K)\) be the full subcategory of \(\Coh_{\XFF}(G_K)\) of \(G_K\)\=/equivariant coherent sheaves whose Harder\--Narasimhan slopes are all \(< 0\).
	Note that the category \(\Coh_{\XFF}^{< 0}(G_K)\) is an exact subcategory of \(\Coh_{\XFF}(G_K)\).
	\item Let \(\Coh_{\XFF}^{0}(G_K)\) be the full subcategory of \(\Coh_{\XFF}(G_K)\) composed of the \(G_K\)\=/equivariant coherent sheaves semistable of slope \(0\).
	Note that the category \(\Coh_{\XFF}^{0}(G_K)\) is an abelian subcategory of \(\Coh_{\XFF}(G_K)\).
	\item Let \(\Coh_{\XFF}^{+\infty}(G_K)\) be the full subcategory of \(\Coh_{\XFF}(G_K)\) of \(G_K\)\=/equivariant torsion coherent sheaves.
	Note that the category \(\Coh_{\XFF}^{+\infty}(G_K)\) is an abelian subcategory of \(\Coh_{\XFF}(G_K)\).
\end{itemize}

The classification of coherent sheaves over \(\XFF\), Theorem~\ref{theorem:FF_classification_perfectoid}, implies the following proposition.

\begin{proposition} \label{proposition:torsion_pair_coh}
	The tuple \((\Coh_{\XFF}^{\geq 0}(G_K),\Coh_{\XFF}^{< 0}(G_K))\) is a torsion pair on \(\Coh_{\XFF}(G_K)\).
\end{proposition}

\begin{theorem}[Fargues\--Fontaine, Fontaine] \label{theorem:almost_coh}
	There are functors
	\[
		\begin{split}
			\Coh_{\XFF}(G_K) & \rightarrow \Cc^{\geq 0}(G_K) \\
			\Fc & \mapsto \HH^0(\XFF,\Fc)
		\end{split}
	\]
	and
	\[
		\begin{split}
			\Coh_{\XFF}(G_K) & \rightarrow \Cc^{< 0}(G_K) \\
			\Fc & \mapsto \HH^1(\XFF,\Fc)
		\end{split}
	\]
	which induce the following equivalence of categories.
	\begin{enumerate}
		\item The functor \(\Fc \mapsto \HH^0(\XFF,\Fc)\) induces an equivalence of exact categories
		\[
			\Coh_{\XFF}^{\geq 0}(G_K) \similarrightarrow \Cc^{\geq 0}(G_K).
		\]
		\item The functor \(\Fc \mapsto \HH^0(\XFF,\Fc)\) induces an equivalence of abelian categories
		\[
			\Coh_{\XFF}^{0}(G_K) \similarrightarrow \Cc^{0}(G_K),
		\]
		of which the functor
		\[
			\begin{split}
				\Cc^{0}(G_K) & \similarrightarrow \Coh_{\XFF}^{0}(G_K) \\
				V & \mapsto \Ec(V) = \Oc_{\XFF} \otimes_{\Qp} V = (\Be \otimes_{\Qp} V, \BdR^+ \otimes_{\Qp} V)
			\end{split}
		\]
		is a quasi-inverse.
		\item The functor \(\Fc \mapsto \HH^0(\XFF,\Fc)\) induces an equivalence of abelian categories
		\[
			\Coh_{\XFF}^{+\infty}(G_K) \similarrightarrow \Cc^{+\infty}(G_K),
		\]
		of which the functor
		\[
			\begin{split}
				\Cc^{+\infty}(G_K) & \similarrightarrow \Coh_{\XFF}^{+\infty}(G_K) \\
				M_\dR^+ & \mapsto (0,M_\dR^+)
			\end{split}
		\]
		is a quasi-inverse.
		\item The functor \(\Fc \mapsto \HH^1(\XFF,\Fc)\) induces an equivalence of exact categories
		\[
			\Coh_{\XFF}^{< 0}(G_K) \similarrightarrow \Cc^{< 0}(G_K).
		\]
	\end{enumerate}
\end{theorem}

\begin{remark}
	While the global sections functor does not extend to an equivalence of categories between \(\Coh_{\XFF}(G_K)\) and \(\Cc(G_K)\), Fontaine has proved that it induces an equivalence of triangulated categories between the bounded derived categories
	\begin{equation} \label{eq:equiv_derived}
		\DD^b(\Coh_{\XFF}(G_K)) \similarrightarrow \DD^b(\Cc(G_K)).
	\end{equation}
	Fontaine has also proved that the categories \(\Coh_{\XFF}(G_K)\) and \(\Cc(G_K)\) can be recovered from each other as follows.
	\begin{itemize}
		\item The torsion pair \((\Coh_{\XFF}^{\geq 0}(G_K),\Coh_{\XFF}^{< 0}(G_K))\) on \(\Coh_{\XFF}(G_K)\) induces a \(t\)\=/structure on \(\DD^b(\Coh_{\XFF}(G_K))\) whose abelian heart is naturally equivalent to \(\Cc(G_K)\) via the equivalence~\eqref{eq:equiv_derived}.
		\item The torsion pair \((\Cc^{< 0}(G_K),\Cc^{\geq 0}(G_K))\) on \(\Cc(G_K)\) induces a \(t\)\=/structure on \(\DD^b(\Cc(G_K))\) whose abelian heart is naturally equivalent to \(\Coh_{\XFF}(G_K)\) via the equivalence~\eqref{eq:equiv_derived}.
	\end{itemize}
\end{remark}

We also set the following subcategories of \(\Coh_{\XFF}(G_K)\) and \(\Cc(G_K)\).
\begin{itemize}
	\item Let \(\Coh_{\XFF}^{> 0}(G_K)\) be the full subcategory of \(\Coh_{\XFF}(G_K)\) of \(G_K\)\=/equivariant coherent sheaves whose Harder\--Narasimhan slopes are all \(> 0\).
	Note that the category \(\Coh_{\XFF}^{> 0}(G_K)\) is a full subcategory of \(\Coh_{\XFF}^{\geq 0}(G_K)\), and that the category \(\Coh_{\XFF}^{+\infty}(G_K)\) is a full subcategory of \(\Coh_{\XFF}^{> 0}(G_K)\).
	\item Let \(\Cc^{> 0}(G_K)\) be the subcategory of \(\Cc(G_K)\) equivalent to the category \(\Coh_{\XFF}^{> 0}(G_K)\) via the equivalence of categories \(\Coh_{\XFF}^{\geq 0}(G_K) \similarrightarrow \Cc^{\geq 0}(G_K)\) from Theorem~\ref{theorem:almost_coh}.
	Note that the category \(\Cc^{+\infty}(G_K)\) is a full subcategory of \(\Cc^{> 0}(G_K)\).
\end{itemize}

\subsection{Almost \(\Cp\)-representations and \(p\)-adic representations}
\label{subsec:almost_rep}

\begin{definition} \label{definition:neq0}
	Let \(\Cc^{\neq 0}(G_K)\) be the full subcategory of \(\Cc(G_K)\) of almost \(\Cp\)\=/representation \(X\) of \(G_K\) such that for all \(p\)\=/adic representation \(V\) of \(G_K\), we have
	\[
		\Hom_{\Cc(G_K)}(X,V) = 0.
	\]
\end{definition}

\begin{proposition} \label{proposition:almost_torsion_relation}
	The tuple \((\Cc^{\neq 0}(G_K),\Cc^0(G_K))\) is a torsion pair on \(\Cc(G_K)\).
	Moreover, if \(X\) is an almost \(\Cp\)\=/representation of \(G_K\), then there exists a commutative diagram in \(\Cc(G_K)\) whose columns and rows are exact which is unique up to isomorphism
	\[
		\begin{tikzcd}
				 &                                    & 0 \ar{d}                 & 0 \ar{d}                 &   \\
                        0 \ar{r} & \ar{r} X^{<0} \ar{r} \ar[equal]{d} & X^{\neq 0} \ar{r} \ar{d} & X^{> 0} \ar{r} \ar{d}    & 0 \\
                        0 \ar{r} & \ar{r} X^{<0} \ar{r}               & X \ar{r} \ar{d}          & X^{\geq 0} \ar{r} \ar{d} & 0 \\
                                 &                                    & X^0 \ar[equal]{r} \ar{d} & X^0 \ar{d}               &   \\
                                 &                                    & 0                        & 0                        & ,
		\end{tikzcd}
	\]
	with \(X^{< 0}\) an object of \(\Cc^{< 0}(G_K)\), \(X^{> 0}\) an object of \(\Cc^{> 0}(G_K)\), \(X^{\geq 0}\) an object of \(\Cc^{\geq 0}(G_K)\), \(X^0\) an object of \(\Cc^0(G_K)\), and \(X^{\neq 0}\) an object of \(\Cc^{\neq 0}(G_K)\).
\end{proposition}
\begin{proof}
	By Definition~\ref{definition:neq0} of the subcategory \(\Cc^{\neq 0}(G_K)\), the first property in the Definition~\ref{definition:torsion_pair} of a torsion pair is satisfied.
	We prove the second one.

	Let \(X\) be an almost \(\Cp\)\=/representation of \(G_K\).
	Since  \((\Cc^{< 0}(G_K),\Cc^{\geq 0}(G_K))\) is a torsion pair on \(\Cc(G_K)\) by Proposition~\ref{proposition:effective_coeffective}, there exits a short exact sequence
	\begin{equation} \label{eq:torsion_pair_0}
		0 \rightarrow X^{< 0} \rightarrow X \rightarrow X^{\geq 0} \rightarrow 0,
	\end{equation}
	with \(X^{< 0}\) a coeffective almost \(\Cp\)\=/representation of \(G_K\), and \(X^{\geq 0}\) an effective almost \(\Cp\)\=/representation of \(G_K\).
	By Theorem~\ref{theorem:almost_coh}, there exists a sheaf \(\Fc^{\geq 0}\) which is an object of \(\Coh_{\XFF}^{\geq 0}(G_K)\) and an isomorphism of almost \(\Cp\)\=/representations of \(G_K\)
	\[
		\HH^0(\XFF,\Fc^{\geq 0}) \similarrightarrow X^{\geq 0}.
	\]
	Let
	\begin{equation} \label{eq:HN0}
		0 \rightarrow \Fc^{> 0} \rightarrow \Fc^{\geq 0} \rightarrow \Fc^0 \rightarrow 0,
	\end{equation}
	be the first step of the Harder\--Narasimhan filtration of \(\Fc^{\geq 0}\), with  \(\Fc^{> 0}\) an object of \(\Coh_{\XFF}^{>0}(G_K)\), and \(\Fc^0\) an object of \(\Coh_{\XFF}^{0}(G_K)\).
	By Theorem~\ref{theorem:almost_coh}, we have the almost \(\Cp\)\=/representations of \(G_K\)
	\[
		\begin{split}
			X^{> 0} & = \HH^0(\XFF,\Fc^{> 0}), \\
			X^0     & = \HH^0(\XFF,\Fc^0),
		\end{split}
	\]
	where \(X^{>0}\) is an object of \(\Cc^{>0}(G_K)\), and \(X^0\) is a \(p\)\=/adic representation of \(G_K\).
	By Theorem~\ref{proposition:cohomology_HN}, the group \(\HH^1(\XFF,\Fc^{> 0})\) is trivial, and the cohomology of the short exact sequence~\eqref{eq:HN0} gives rise to the short exact sequence
	\begin{equation} \label{eq:almost_HN0}
		0 \rightarrow X^{> 0} \rightarrow X^{\geq 0} \rightarrow X^0 \rightarrow 0.
	\end{equation}
	Let \(X^{\neq 0}\) be the reciprocal image of \(X^{> 0}\) in \(X\) via the short exact sequence~\eqref{eq:torsion_pair_0}.
	The combination of the short exact sequences~\eqref{eq:torsion_pair_0} and~\eqref{eq:almost_HN0} yields the commutative diagram in \(\Cc(G_K)\) whose columns and rows are exact
	\begin{equation} \label{eq:almost_torsion_relation}
		\begin{tikzcd}
				 &                                    & 0 \ar{d}                 & 0 \ar{d}                 &   \\
                        0 \ar{r} & \ar{r} X^{<0} \ar{r} \ar[equal]{d} & X^{\neq 0} \ar{r} \ar{d} & X^{> 0} \ar{r} \ar{d}    & 0 \\
                        0 \ar{r} & \ar{r} X^{<0} \ar{r}               & X \ar{r} \ar{d}          & X^{\geq 0} \ar{r} \ar{d} & 0 \\
                                 &                                    & X^0 \ar[equal]{r} \ar{d} & X^0 \ar{d}               &   \\
                                 &                                    & 0                        & 0                        & .
		\end{tikzcd}
	\end{equation}

	We prove that \(X^{\neq 0}\) is an object of \(\Cc^{\neq 0}(G_K)\).
	Let \(V\) be a \(p\)\=/adic representation of \(G_K\).
	On the one hand, since \(X^{< 0}\) is an object of \(\Cc^{< 0}(G_K)\) and the category of \(p\)\=/adic representation \(\Cc^0(G_K)\) is a subcategory of \(\Cc^{\geq 0}(G_K)\) by Proposition~\ref{proposition:effective_coeffective}, and since the tuple \((\Cc^{< 0}(G_K),\Cc^{\geq 0}(G_K))\) is a torsion pair on \(\Cc(G_K)\) again by Proposition~\ref{proposition:effective_coeffective}, there is no non-trivial map from \(X^{< 0}\) to \(V\).
	On the other hand, the \(G_K\)\=/equivariant vector bundle \(\Ec(V)\) associated with \(V\) is semistable of slope \(0\) by Theorem~\ref{theorem:almost_coh}, therefore, by Proposition~\ref{proposition:torsion_pair_coh} there is no non-trivial map from \(\Fc^{> 0}\) to \(\Ec(V)\), and thus, by the equivalence of categories \(\Coh_{\XFF}^{\geq 0}(G_K) \similarrightarrow \Cc^{\geq 0}(G_K)\) from Theorem~\ref{theorem:almost_coh}, there is no non-trivial map from \(X^{> 0}=\HH^0(\XFF,\Fc^{> 0})\) to \(V = \HH^0(\XFF,\Ec(V))\).
	Therefore, using the short exact sequence
	\[
		0 \rightarrow X^{< 0} \rightarrow X^{\neq 0} \rightarrow X^{> 0} \rightarrow 0
	\]
	extracted from the diagram~\eqref{eq:almost_torsion_relation}, we conclude that there is no non-trivial map from \(X^{\neq 0}\) to \(V\), and hence, \(X^{\neq 0}\) is an object of \(\Cc^{\neq 0}(G_K)\).

	Finally, the existence of the short exact sequence
	\[
		0 \rightarrow X^{\neq 0} \rightarrow X \rightarrow X^0 \rightarrow 0
	\]
	extracted from the diagram~\eqref{eq:almost_torsion_relation} with \(X^{\neq 0}\) an object of \(\Cc^{\neq 0}(G_K)\) and \(X^0\) an object of \(\Cc^0(G_K)\) implies that the second property in the Definition~\ref{definition:torsion_pair} of a torsion pair is satisfied.
\end{proof}

We will also need the following.

\begin{lemma} \label{lemma:relation_coeffective_torsion}
	If \(X^{<0}\) is a coeffective almost \(\Cp\)\=/representation of \(G_K\), then there exists a short exact sequence in \(\Cc(G_K)\)
	\[
		0 \rightarrow Y^{> 0} \rightarrow Z^{+\infty} \rightarrow X^{< 0} \rightarrow 0,
	\]
	with \(Y^{>0}\) an object of \(\Cc^{> 0}(G_K)\), and \(Z^{+\infty}\) an object of \(\Cc^{+ \infty}(G_K)\).
\end{lemma}
\begin{proof}
	By Theorem~\ref{theorem:almost_coh}, there exists an object \(\Fc\) of \(\Coh_{\XFF}^{< 0}(G_K)\) and an isomorphism of almost \(\Cp\)\=/representations
	\[
		\HH^1(\XFF,\Fc) \similarrightarrow X^{<0}.
	\]
	For \(n \in \N\), the Harder\--Narasimhan twist \(\Fc(n)_\HN\) of \(\Fc\) fits into the short exact sequence
	\begin{equation} \label{eq:HNtwist}
		0 \rightarrow \Fc \rightarrow \Fc(n)_\HN \rightarrow \Hc(n) \rightarrow 0,
	\end{equation}
	with \(\Hc(n)\) the torsion \(G_K\)\=/equivariant coherent sheaf \((0, t^{-n} \Fc_\dR^+/ \Fc_\dR^+)\).
	Moreover, for \(n\) sufficiently large, the Harder\--Narasimhan slopes of \(\Fc(n)_\HN\) are all \(>0\) by Proposition~\ref{proposition:HNtwist}.
	Hence, by Theorem~\ref{proposition:cohomology_HN} and Theorem~\ref{theorem:almost_coh}, the short exact sequence~\eqref{eq:HNtwist} induces a short exact sequence
	\[
		0 \rightarrow \HH^0(\XFF,\Fc(n)_\HN) \rightarrow \HH^0(\XFF,\Hc(n)) \rightarrow \HH^1(\XFF,\Fc) \rightarrow 0,
	\]
	and \(\HH^0(\XFF,\Fc(n)_\HN)\) is an object of \(\Cc^{> 0}(G_K)\), and \(\HH^0(\XFF,\Hc(n))\) is an object of \(\Cc^{+\infty}(G_K)\).
\end{proof}

\subsection{de Rham vector bundles}
\label{subsec:Bun_dR}

We briefly recall the definition of de Rham vector bundles over the Fargues\--Fontaine curve~\cite{FarguesFontaine2018}.

Let \(\Mod_K\) be the category of finite dimensional \(K\)\=/vector spaces.
There is a functor
\begin{equation} \label{eq:BdR_Vec}
	\begin{split}
		\Rep_{\BdR}(G_K) & \rightarrow \Mod_K \\
		M_\dR & \mapsto M_\dR^{G_K},
	\end{split}
\end{equation}
which admits a right adjoint
\begin{equation} \label{eq:Vec_BdR}
	\begin{split}
		\Mod_K & \rightarrow \Rep_{\BdR}(G_K) \\
		D & \mapsto \BdR \otimes_K D.
	\end{split}
\end{equation}

A \(\BdR\)\=/representation \(M_\dR\) of \(G_K\) is \emph{flat} if \(\dim_{\BdR} M_\dR = \dim_K M_\dR^{G_K}\).
Let \(\Rep_{\BdR}^\fl(G_K)\) be the full subcategory of \(\Rep_{\BdR}(G_K)\) of flat \(\BdR\)\=/representation of \(G_K\).
Then the functor~\eqref{eq:BdR_Vec} induces an equivalence of categories
\[
	\Rep_{\BdR}^\fl(G_K) \similarrightarrow \Mod_K,
\]
of which the functor~\eqref{eq:Vec_BdR} is a quasi-inverse.

Let \(\Fil_K\) be the category of filtered \(K\)\=/vector spaces, that is, the category of finite dimensional \(K\)\=/vector space equipped with a decreasing exhaustive and separated filtration by \(K\)\=/vector subspaces.
The weights of a filtered \(K\)\=/vector space \((D,\Fil D)\) are the integers \(n \in \Z\) such that \(\Fil^{-n} D/\Fil^{-n+1} D \neq 0\), and the multiplicity of a weight \(n\) is the dimension \(\dim_K \Fil^{-n} D/\Fil^{-n+1} D\).

Let \(\Rep_{\BdR^+}^\free(G_K)\) be the full subcategory of \(\Rep_{\BdR^+}(G_K)\) of \(\BdR^+\)\=/representation of \(G_K\) whose underlying \(\BdR^+\)\=/module is free.
There is a functor
\begin{equation} \label{eq:rep_fil}
	\begin{split}
		\Rep_{\BdR^+}^\free(G_K) & \rightarrow \Fil_K \\
		M_\dR^+ & \mapsto (M_\dR^{G_K}, \{(t^n M_\dR^+)^{G_K}\}_{n \in \Z})
	\end{split}
\end{equation}
which admits a right adjoint
\begin{equation} \label{eq:fil_rep}
	\begin{split}
		\Fil_K & \rightarrow \Rep_{\BdR^+}^\free(G_K) \\
		(D,\Fil D) & \mapsto \sum_{n \in \Z} \Fil^n \BdR \otimes_K \Fil^{-n} D.
	\end{split}
\end{equation}

A free \(\BdR^+\)\=/representation \(M_\dR^+\) of \(G_K\) is \emph{generically flat} if the \(\BdR\)\=/representation \(M_\dR\) is flat.
Let \(\Rep_{\BdR^+}^\gfl(G_K)\) be the full subcategory of \(\Rep_{\BdR^+}^\free(G_K)\) of generically flat \(\BdR^+\)\=/representation of \(G_K\).

\begin{theorem}[Fargues\--Fontaine] \label{theorem:genflatFil}
	The functor~\eqref{eq:rep_fil} induces an equivalence of categories
	\[
		\Rep_{\BdR^+}^\gfl(G_K) \similarrightarrow \Fil_K,
	\]
	of which the functor~\eqref{eq:fil_rep} is a quasi-inverse.
\end{theorem}

There is a functor
\[
	\begin{split}
		\DDb_\dR: \Bun_{\XFF}(G_K) & \rightarrow \Fil_K \\
			\Ec & \mapsto (\DDb_\dR(\Ec),\Fil \DDb_\dR(\Ec)),
	\end{split}
\]
defined as the composition of the functor~\eqref{eq:rep_fil} with the functor
\[
	\begin{split}
		\Bun_{\XFF}(G_K) & \rightarrow \Rep_{\BdR^+}^\free(G_K) \\
		\Ec & \mapsto \Ec_\dR^+.
	\end{split}
\]

A \(G_K\)\=/equivariant vector bundle \(\Ec\) is \emph{de Rham} if \(\Ec_\dR^+\) is generically flat, or equivalently, if \(\dim_K \DDb_\dR(\Ec) = \rank \Ec\).
The \emph{Hodge\--Tate weights} of a de Rham vector bundle \(\Ec\) are the weights of \(\DDb_\dR(\Ec)\).
We denote by \(\Bun_{\XFF}(G_K)_\dR\) the full subcategory of \(\Bun_{\XFF}(G_K)\) of \(G_K\)\=/equivariant de Rham vector bundles over \(\XFF\).

\begin{remark} \label{remark:Bun_dR}
	The definition of de Rham vector bundles and Proposition~\ref{proposition:char} implies the following characterisation of de Rham vector bundles.
	A \(G_K\)\=/equivariant coherent sheaf \(\Ec\) is a de Rham vector bundle if and only if \(\Ec_\dR^+\) is generically flat.
\end{remark}

We have the following~\cite[Proposition~1.7.5]{Ponsinet2024:1}.

\begin{proposition} \label{proposition:dR}
	Let
	\[
		0 \rightarrow \Ec^\prime \rightarrow \Ec \rightarrow \Ec^\dprime \rightarrow 0
	\]
	be a short exact sequence of \(G_K\)\=/equivariant vector bundles over \(\XFF\).
	If \(\Ec\) is de Rham, then \(\Ec^\prime\) and \(\Ec^\dprime\) are de Rham.
	Moreover, the set of the Hodge\--Tate weights of \(\Ec\) is the union of the sets of the Hodge\--Tate weights of \(\Ec^\prime\) and \(\Ec^\dprime\).
\end{proposition}

\begin{remark}
	The composition of functors
	\[
		\begin{split}
			\Rep_{\Qp}(G_K) & \rightarrow \Fil_K \\
			V & \mapsto \DDb_\dR(\Ec(V))
		\end{split}
	\]
	is the usual \(\DDb_\dR\) functor defined by Fontaine~\cite{Fontaine1994:III}.
	In particular, a \(p\)\=/adic representation \(V\) is de Rham if and only if the vector bundle \(\Ec(V)\) is de Rham.
	We will write \(\DDb_\dR(V)\) instead of \(\DDb_\dR(\Ec(V))\).
\end{remark}

\section{Cohomology of perfectoid fields}
\label{sec:perfectoid}

\subsection{Cohomology of almost \(\Cp\)-representations}
\label{subsec:cohom_almost}

Let \(L\) be an algebraic extension of \(K\).
If \(\hat{L}\) is a perfectoid field, as a consequence of the classification of \(G_L\)\=/equivariant coherent sheaves over \(\XFF\) already mentioned in Theorem~\ref{theorem:FF_classification_perfectoid}, Fargues and Fontaine~\cite[Remarque~9.4.2]{FarguesFontaine2018} have obtained the following.

\begin{theorem}[Fargues\--Fontaine] \label{theorem:FF_perfectoid}
	Let \(\Fc\) be a \(G_L\)\=/equivariant coherent sheaf over \(\XFF\) whose Harder\--Narasimhan slopes are all \(> 0\).
	If \(\hat{L}\) is a perfectoid field, then
	\[
		\Ext^1_{\Coh_{\XFF}(G_K)}(\Oc_{\XFF},\Fc) \similarrightarrow \HH^1(L,\HH^0(\XFF,\Fc)) = 0.
	\]
\end{theorem}

The combination of Theorem~\ref{theorem:almost_coh} and Theorem~\ref{theorem:FF_perfectoid} immediately implies the following.

\begin{corollary} \label{corollary:FF_perfectoid_almost}
	Let \(X^{>0}\) be an almost \(\Cp\)\=/representation which is an object of the subcategory \(\Cc^{>0}(G_K)\).
	If \(\hat{L}\) is a perfectoid field, then the group \(\HH^1(L,X^{>0})\) is trivial.
\end{corollary}

We will use the following repeatedly.
\begin{proposition} \label{proposition:cohom_dim_perfectoid}
	The \(p\)\=/cohomological dimension of a perfectoid field of residue characteristic \(p\) is \(\leq 1\).
\end{proposition}
\begin{proof}
	Let \(k\) be a perfectoid field of residue characteristic \(p\).
	On the one hand, the tilt \(k^\flat\) of \(k\) is a perfectoid field of characteristic \(p\) whose absolute Galois group \(G_{k^\flat}\) is canonically isomorphic to \(G_k\) (\cite[\S 3]{Scholze2012}).
	On the other hand, the \(p\)\=/cohomological dimension of a field characteristic \(p\) is \(\leq 1\) (\cite[II \S 2.2 Proposition~3]{Serre1994}).
\end{proof}

\begin{lemma} \label{lemma:cohom_n}
	Let \(X\) be a \(p\)\=/adic Banach representation of \(G_K\), and let \(\Xc\) be a \(G_K\)\=/stable lattice in \(X\).
	If the \(p\)\=/cohomological dimension of \(L\) is \(\leq 1\), then, for each integer \(n > 1\), the groups \(\HH^n(L,\Xc)\) and \(\HH^n(L,X)\) are trivial.
\end{lemma}
\begin{proof}
	Recall~\cite[\S 2]{Jannsen1988} that, since \(\Xc\) is complete and separated for the \(p\)\=/adic topology, for each \(n \in \N\), there is a short exact sequence
	\[
		0 \rightarrow {\varprojlim}^1 \HH^{n-1}(L,\Xc/p^i \Xc) \rightarrow \HH^n(L,\Xc) \rightarrow \varprojlim \HH^n(L,\Xc/p^i \Xc) \rightarrow 0,
	\]
	where we set \(\HH^{-1}(L,\Xc/p^i \Xc) = 0\).
	Moreover, for each \(i \in \N\), the short exact sequence
	\[
		0 \rightarrow p^i \Xc/p^{i+1} \Xc \rightarrow \Xc/p^{i+1} \Xc \rightarrow \Xc/p^i \Xc \rightarrow 0
	\]
	induces an exact sequence
	\[
		\HH^1(L,\Xc/p^{i+1}\Xc) \rightarrow \HH^1(L,\Xc/p^i \Xc) \rightarrow \HH^{2}(L,p^i \Xc/p^{i+1} \Xc).
	\]
	By hypothesis, for each integers \(n > 1\) and \(i \in \N\), the groups \(\HH^n(L,\Xc/p^i \Xc)\) and \(\HH^n(L,p^i\Xc/p^{i+1} \Xc)\) are trivial, which implies that:
	\begin{itemize}
		\item for each integer \(n > 1\), the group \(\varprojlim \HH^n(L,\Xc/p^i \Xc)\) is trivial,
		\item for each integer \(n > 2\), the group \({\varprojlim}^1 \HH^{n-1}(L,\Xc/p^i \Xc)\) is trivial,
		\item for each \(i \in \N\), the map \(\HH^1(L,\Xc/p^{i+1}\Xc) \rightarrow \HH^1(L,\Xc/p^i \Xc)\) is surjective, and thus, the group \({\varprojlim}^1 \HH^1(L,\Xc/p^i \Xc)\) is also trivial.
	\end{itemize}
	Therefore, for each integer \(n > 1\), the group \(\HH^n(L,\Xc)\) is trivial.

	Since \(X/\Xc\) is discrete, the short exact sequence
	\[
		0 \rightarrow \Xc \rightarrow X \rightarrow X/\Xc \rightarrow 0
	\]
	induces a long exact sequence
	\[
		\cdots \rightarrow \HH^n(L,\Xc) \rightarrow \HH^n(L,X) \rightarrow \HH^n(L,X/\Xc) \rightarrow \HH^{n+1}(L,\Xc) \rightarrow \cdots.
	\]
	For each integer \(n > 1\), we have proved that the group \(\HH^n(L,\Xc)\) is trivial, and by hypothesis, the group \(\HH^n(L,X/\Xc)\) is trivial.
	Therefore, for each integer \(n > 1\), the group \(\HH^n(L,X)\) is trivial.
\end{proof}

\begin{remark} \label{remark:les_almost}
	A short exact sequence in \(\Cc(G_K)\)
	\begin{equation} \label{eq:XYZ_top}
		0 \rightarrow Z \rightarrow X \rightarrow Y \rightarrow 0
	\end{equation}
	induces a long exact sequence
	\[
		\cdots \rightarrow \HH^n(L,Z) \rightarrow \HH^n(L,X) \rightarrow \HH^n(L,Y) \rightarrow \HH^{n+1}(L,Z) \rightarrow \cdots.
	\]
	Indeed, by Theorem~\ref{theorem:almost_abelian}, the category \(\Cc(G_K)\) is an abelian subcategory of \(\Bc(G_K)\).
	In particular, each morphism of \(p\)\=/adic Banach spaces in the sequence~\eqref{eq:XYZ_top} is strict, and hence, there exists a \(\Qp\)\=/linear and continuous section of the surjective morphism \(X \rightarrow Y\) (see for instance~\cite[Proposition~I.1.5~(iii)]{Colmez1998}).
	Therefore, the short exact sequence~\eqref{eq:XYZ_top} induces long exact sequences in Galois cohomology~\cite[\S 2]{Tate1976}.
\end{remark}

\begin{proposition} \label{proposition:cohomology_almost_perfectoid}
	Let \(X\) be an almost \(\Cp\)\=/representation of \(G_K\).
	If \(\hat{L}\) is a perfectoid field, then the short exact sequence
	\[
		0 \rightarrow X^{\neq 0} \rightarrow X \rightarrow X^0 \rightarrow 0
	\]
	induces an isomorphism
	\[
		\HH^1(L,X) \similarrightarrow \HH^1(L,X^0).
	\]
\end{proposition}
\begin{proof}
	The combination of Proposition~\ref{proposition:cohom_dim_perfectoid}, Lemma~\ref{lemma:cohom_n} and Remark~\ref{remark:les_almost} yields an exact sequence
	\[
		\HH^1(L,X^{\neq 0}) \rightarrow \HH^1(L,X) \rightarrow \HH^1(L,X^0) \rightarrow 0.
	\]
	We prove that the group \(\HH^1(L,X^{\neq 0})\) is trivial.
	By Proposition~\ref{proposition:almost_torsion_relation}, there exists a short exact sequence
	\begin{equation} \label{eq:neq0bis}
		0 \rightarrow X^{< 0} \rightarrow X^{\neq 0} \rightarrow X^{> 0} \rightarrow 0,
	\end{equation}
	with \(X^{< 0}\) an object of \(\Cc^{< 0}(G_K)\), and \(X^{> 0}\) an object of \(\Cc^{> 0}(G_K)\).
	By Proposition~\ref{proposition:cohom_dim_perfectoid}, Lemma~\ref{lemma:cohom_n} and Remark~\ref{remark:les_almost}, the short exact sequence~\eqref{eq:neq0bis} induces an exact sequence
	\begin{equation} \label{eq:cohom_neq0}
		\HH^1(L,X^{< 0}) \rightarrow \HH^1(L,X^{\neq 0}) \rightarrow \HH^1(L,X^{>0}) \rightarrow 0.
	\end{equation}
	Moreover, by Lemma~\ref{lemma:relation_coeffective_torsion}, there exists a short exact sequence
	\begin{equation} \label{eq:coeffective_torsion}
		0 \rightarrow Y^{> 0} \rightarrow Z^{+\infty} \rightarrow X^{< 0} \rightarrow 0,
	\end{equation}
	with \(Y^{>0}\) an object of \(\Cc^{> 0}(G_K)\), and \(Z^{+\infty}\) an object of \(\Cc^{+ \infty}(G_K)\).
	Again by Proposition~\ref{proposition:cohom_dim_perfectoid}, Lemma~\ref{lemma:cohom_n} and Remark~\ref{remark:les_almost}, the short exact sequence~\eqref{eq:coeffective_torsion} induces an exact sequence
	\begin{equation} \label{eq:cohom_coeffective_torsion}
		\HH^1(L,Y^{> 0}) \rightarrow \HH^1(L,Z^{+\infty}) \rightarrow \HH^1(L,X^{<0}) \rightarrow 0.
	\end{equation}
	Finally, by Corollary~\ref{corollary:FF_perfectoid_almost}, the groups \(\HH^1(L,Z^{+\infty})\) and \(\HH^1(L,X^{>0})\) are trivial.
	Thus, using the exact sequences~\eqref{eq:cohom_neq0} and~\eqref{eq:cohom_coeffective_torsion}, we conclude that the groups \(\HH^1(L,X^{< 0})\) and \(\HH^1(L,X^{\neq 0})\) are trivial.
\end{proof}

\begin{lemma} \label{lemma:surjection}
	Let
	\[
		0 \rightarrow Z \rightarrow X \rightarrow Y \rightarrow 0
	\]
	be a short exact sequence in \(\Cc(G_K)\), and let \(\Zc\) be a \(G_K\)\=/stable lattice in \(Z\).
	If the \(p\)\=/cohomological dimension of \(L\) is \(\leq 1\), then the short exact sequence of topological \(G_K\)\=/modules
	\[
		0 \rightarrow \Zc \rightarrow X \rightarrow X/\Zc \rightarrow 0
	\]
	induces a surjective map
	\[
		\HH^1(L,X) \rightarrow \HH^1(L,X/\Zc) \rightarrow 0
	\]
\end{lemma}
\begin{proof}
	The quotient map \(X \rightarrow X/Z\) admits a continuous section.
	Indeed, by Remark~\ref{remark:les_almost}, there exists an isomorphism of \(p\)\=/adic Banach spaces \(X \similarrightarrow Z \oplus Y\), which induces an isomorphism of topological groups \(X/\Zc \similarrightarrow Z/\Zc \oplus Y\).
	Since the group \(Z/\Zc\) is discrete, the quotient map \(Z \oplus Y \rightarrow Z/\Zc \oplus Y\) admits a continuous section.

	Therefore, the short exact sequence of topological \(G_K\)\=/modules
	\[
		0 \rightarrow \Zc \rightarrow X \rightarrow X/\Zc \rightarrow 0
	\]
	induces an exact sequence
	\[
		\HH^1(L,X) \rightarrow \HH^1(L,X/\Zc) \rightarrow \HH^2(L,\Zc),
	\]
	and the statement follows from Lemma~\ref{lemma:cohom_n}.
\end{proof}

The combination of Proposition~\ref{proposition:cohom_dim_perfectoid}, Proposition~\ref{proposition:cohomology_almost_perfectoid} and Lemma~\ref{lemma:surjection} implies the following.

\begin{corollary} \label{corollary:cohomology_almost_perfectoid_lattice}
	Let
	\[
		0 \rightarrow Z \rightarrow X \rightarrow Y \rightarrow 0
	\]
	be a short exact sequence in \(\Cc(G_K)\), and let \(\Zc\) be a \(G_K\)\=/stable lattice in \(Z\).
	Let \(\Zc^{(0)}\) be the image of \(\Zc\) in \(X^0\), and let \(\Zc^{(\neq 0)} = \Zc \cap X^{\neq 0}\).
	If \(\hat{L}\) is a perfectoid field, then the short exact sequence of topological \(G_K\)\=/modules
	\[
		0 \rightarrow X^{\neq 0}/\Zc^{(\neq 0)} \rightarrow X/\Zc \rightarrow X^0/\Zc^{(0)} \rightarrow 0
	\]
	induces an isomorphism
	\[
		\HH^1(L,X/\Zc) \similarrightarrow \HH^1(L,X^0/\Zc^{(0)}).
	\]
	Moreover, for each integer \(n > 1\), the group \(\HH^n(L,X/\Zc)\) is trivial.
\end{corollary}

\subsection{Cohomology of maximal discrete Galois submodules}
\label{subsec:delta}

\begin{notation}
	If \(M\) is a topological \(G_K\)\=/module, then we denote by \(M_\delta\) the discrete \(G_K\)\=/module
	\[
		M_\delta = \varinjlim_{\res,K^\prime} \HH^0(K^\prime,M),
	\]
	where \(K^\prime\) runs over all the finite extensions of \(K\), and the transition morphisms are the restriction maps.
\end{notation}

Let
\[
	0 \rightarrow Z \rightarrow X \rightarrow Y \rightarrow 0
\]
be a short exact sequence in \(\Cc(G_K)\), and let \(\Zc\) be a \(G_K\)\=/stable lattice in \(Z\).

\begin{lemma} \label{lemma:delta_ses}
	The short exact sequence of topological \(G_K\)\=/modules
	\[
		0 \rightarrow Z/\Zc \rightarrow X/\Zc \rightarrow Y \rightarrow 0
	\]
	induces a short exact sequence of discrete \(G_K\)\=/modules
	\[
		0 \rightarrow Z/\Zc \rightarrow (X/\Zc)_\delta \rightarrow Y_\delta \rightarrow 0.
	\]
	In particular, there is a commutative diagram of topological \(G_K\)\=/modules with exact rows
	\[
		\begin{tikzcd}
			0 \ar{r} & Z/\Zc \ar{r} & X/\Zc \ar{r} & Y \ar{r} & 0 \\
			0 \ar{r} & Z/\Zc \ar{r} \ar[equal]{u} & (X/\Zc)_\delta \ar{r} \ar{u} & Y_\delta \ar{r} \ar{u} & 0.
		\end{tikzcd}
	\]
\end{lemma}
\begin{proof}
	The short exact sequence of topological \(G_K\)\=/modules
	\[
		0 \rightarrow Z/\Zc \rightarrow X/\Zc \rightarrow Y \rightarrow 0
	\]
	induces a long exact sequence for each finite extension \(K^\prime\) of \(K\)
	\begin{equation} \label{eq:les_XZ}
		0 \rightarrow \HH^0(K^\prime,Z/\Zc) \rightarrow \HH^0(K^\prime,X/\Zc) \rightarrow \HH^0(K^\prime,Y) \rightarrow \HH^1(K^\prime,Z/\Zc) \rightarrow \cdots.
	\end{equation}
	The long exact sequences~\eqref{eq:les_XZ} yields the long exact sequence
	\[
		0 \rightarrow (Z/\Zc)_\delta \rightarrow (X/\Zc)_\delta \rightarrow Y_\delta \rightarrow \varinjlim_{\res,K^\prime} \HH^1(K^\prime,Z/\Zc) \rightarrow \cdots,
	\]
	where \(K^\prime\) runs over all the finite extensions of \(K\), and the transition morphisms are the restriction maps.
	Since \(Z/\Zc\) is a discrete \(G_K\)\=/module, we have~\cite[I \S 2.2 Proposition~8]{Serre1994}
	\[
		\varinjlim_{\res,K^\prime} \HH^n(K^\prime,Z/\Zc) =
		\begin{cases}
			(Z/\Zc)_\delta = Z/\Zc, & \text{if } n = 0, \\
			0, & \text{if } n > 0.
		\end{cases}
	\]
\end{proof}

Let \(L\) be an algebraic extension of \(K\).

\begin{lemma} \label{lemma:delta_les}
	The short exact sequence of discrete \(G_K\)\=/modules
	\[
		0 \rightarrow Z/\Zc \rightarrow (X/\Zc)_\delta \rightarrow Y_\delta \rightarrow 0
	\]
	induces an exact sequence
	\[
		\begin{tikzcd}
			0 \ar{r} & \HH^0(L,Z/\Zc) \ar{r} & \HH^0(L,(X/\Zc)_\delta) \ar[phantom, ""{coordinate, name=Z}]{d} \ar{r} & \HH^0(L,Y_\delta) \ar[rounded corners, to path={ -- ([xshift=2em]\tikztostart.east) |- (Z) [near end]\tikztonodes -| ([xshift=-2em]\tikztotarget.west) -- (\tikztotarget)}]{dll} \\
			& \HH^1(L,Z/\Zc) \ar{r} & \HH^1(L,(X/\Zc)_\delta) \ar{r} & 0                                        & ,
		\end{tikzcd}
	\]
	and, for each integer \(n > 1\), an isomorphism
	\[
		\HH^n(L,Z/\Zc) \similarrightarrow \HH^n(L,(X/\Zc)_\delta).
	\]
\end{lemma}
\begin{proof}
	The short exact sequence of discrete \(G_K\)\=/modules
	\[
		0 \rightarrow Z/\Zc \rightarrow (X/\Zc)_\delta \rightarrow Y_\delta \rightarrow 0
	\]
	induces a long exact sequence
	\[
		\cdots \rightarrow \HH^n(L,Z/\Zc) \rightarrow \HH^n(L,(X/\Zc)_\delta) \rightarrow \HH^n(L,Y_\delta) \rightarrow \HH^{n+1}(L,Z/\Zc) \rightarrow \cdots.
	\]
	Furthermore, since \(Y_\delta\) is a uniquely divisible discrete \(G_K\)\=/module, for each integer \(n \geq 1\), the group \(\HH^n(L,Y_\delta)\) is trivial since it is torsion~\cite[I \S 2.2 Corollaire 3]{Serre1994} and uniquely divisible.
\end{proof}

By Lemma~\ref{lemma:delta_ses} and Lemma~\ref{lemma:delta_les}, there exists a morphism
\[
	\xi : \HH^1(L,(X/\Zc)_\delta) \rightarrow \HH^1(L,X/\Zc),
\]
and a commutative diagram with exact rows
\begin{equation} \label{eq:cohom_comparison_delta}
	\begin{tikzcd}
		Y^{G_L} \ar{r} & \HH^1(L,Z/\Zc) \ar{r} & \HH^1(L,X/\Zc) \ar{r} & \HH^1(L,Y) \\
		Y_\delta^{G_L} \ar{r} \ar{u} & \HH^1(L,Z/\Zc) \ar{r} \ar[equal]{u} & \HH^1(L,(X/\Zc)_\delta) \ar{r} \ar{u}{\xi} & 0 \ar{u} .
	\end{tikzcd}
\end{equation}
In particular, the map \(\xi\) induces a morphism
\[
	\HH^1(L,(X/\Zc)_\delta) \rightarrow \Ker\left(\HH^1(L,X/\Zc) \rightarrow \HH^1(L,Y)\right).
\]
Note that \(\Ker\left(\HH^1(L,X/\Zc) \rightarrow \HH^1(L,Y)\right)\) is the torsion subgroup of \(\HH^1(L,X/\Zc)\).

We consider \(Y^{G_L}\) endowed with the subspace topology from \(Y\).

\begin{proposition} \label{proposition:density}
	If \((Y_\delta)^{G_L}\) is dense in \(Y^{G_L}\), then the map \(\xi\) induces an isomorphism
	\[
		\HH^1(L,(X/\Zc)_\delta) \similarrightarrow \Ker\left(\HH^1(L,X/\Zc) \rightarrow \HH^1(L,Y)\right).
	\]
\end{proposition}
\begin{proof}
	Recall~\cite[Appendix~A]{Ponsinet2024:1} that if \(G\) is a locally compact and separated topological group, and if \(M\) is a topological \(G\)\=/module, then the compact\--open topology on the continuous cochains induces a structure of topological groups on each abelian group \(\HH^n(G,M)\), \(n \in \N\), which satisfy the following properties.
	\begin{enumerate}
		\item Let \(M^G\) be the submodule of \(G\)\=/invariant elements of \(M\) equipped with the subspace topology from \(M\).
		The canonical isomorphism
		\[
			\HH^0(G,M) \similarrightarrow M^G
		\]
		is an isomorphism of topological abelian groups.
		\item If
		\[
			0 \rightarrow M^\prime \rightarrow M \rightarrow M^\dprime \rightarrow 0
		\]
		is a short exact sequence of topological \(G\)\=/modules such that the topology of \(M^\prime\) is the subspace topology from \(M\), the topology of \(M^\dprime\) is the quotient topology from \(M\), and there exists a continuous section of the projection of \(M\) on \(M^\dprime\) as topological space, then it induces a sequence of topological groups
		\[
			\cdots \rightarrow \HH^n(G,M^\prime) \rightarrow \HH^n(G,M) \rightarrow  \HH^n(G,M^\dprime) \rightarrow \HH^{n+1}(G,M^\prime) \rightarrow \cdots
		\]
		whose underlying sequence of abelian groups is exact.
		\item If \(G\) is compact and \(M\) is a discrete \(G\)\=/module, then, for each \(n \in \N\), the topological group \(\HH^n(L,M)\) is discrete.
	\end{enumerate}

	Therefore, we can consider the top exact sequence
	\[
		Y^{G_L} \rightarrow \HH^1(L,Z/\Zc) \rightarrow \HH^1(L,X/\Zc) \rightarrow \HH^1(L,Y)
	\]
	of the diagram~\eqref{eq:cohom_comparison_delta} as a sequence of topological abelian groups in which the group \(\HH^1(L,Z/\Zc)\) is discrete.
	If \(Y_\delta^{G_L}\) is dense in \(Y^{G_L}\), then, by continuity, the images of \(Y_\delta^{G_L}\) and \(Y^{G_L}\) in the discrete group \(\HH^1(L,Z/\Zc)\) coincides.
	We conclude using the commutativity of the diagram~\eqref{eq:cohom_comparison_delta}.
\end{proof}

The combination of Corollary~\ref{corollary:cohomology_almost_perfectoid_lattice} and Proposition~\ref{proposition:density} yields the following.

\begin{corollary} \label{corollary:delta_perfectoid}
	If \(\hat{L}\) is a perfectoid field and if \((Y_\delta)^{G_L}\) is dense in \(Y^{G_L}\), then there exists a short exact sequence
	\[
		0 \rightarrow \HH^1(L,(X/\Zc)_\delta) \rightarrow \HH^1(L,X^0/\Zc^{(0)}) \rightarrow \HH^1(L,Y^0) \rightarrow 0.
	\]
\end{corollary}

\section{Truncation of the Hodge--Tate filtration}
\label{sec:truncation}

Let \(\Bun_{\XFF}(G_K)_\dR^{\leq 0}\) be the full subcategory of \(\Bun_{\XFF}(G_K)_\dR\) of de Rham \(G_K\)\=/equivariant vector bundles over \(\XFF\) whose Hodge\--Tate weights are all \(\leq 0\).
In~\cite[\S 2]{Ponsinet2024:1}, we have defined and studied the functor
\[
	\begin{split}
		\tau_\dR^{\leq 0} : \Bun_{\XFF}(G_K)_\dR & \rightarrow \Bun_{\XFF}(G_K)_\dR^{\leq 0} \\
		\Ec & \mapsto \Ec_+
	\end{split}
\]
which is a left adjoint to the forgetful functor from \(\Bun_{\XFF}(G_K)_\dR^{\leq 0}\) to \(\Bun_{\XFF}(G_K)_\dR\), and which associates with a vector bundle \(\Ec = (\Ec_\e, \Ec_\dR^+, \iota_{\Ec})\) the vector bundle
\[
	\Ec_+ = (\Ec_\e, \Ec_\dR^+ + \BdR^+ \otimes_K \DDb_\dR(\Ec), \iota_{\Ec}).
\]

Let \(n \geq 1\) be an integer.
Let \(\Bun_{\XFF}(G_K)_\dR^{\leq 0,>n}\) be the full subcategory of \(\Bun_{\XFF}(G_K)_\dR\) of de Rham \(G_K\)\=/equivariant vector bundles over \(\XFF\) whose Hodge\--Tate weights are all in the set \(\Z \setminus [1,n]\).
In this section~\ref{sec:truncation}, we will define and study the functor
\[
	\begin{split}
		\tau_\dR^{\leq 0,>n} : \Bun_{\XFF}(G_K)_\dR & \rightarrow \Bun_{\XFF}(G_K)_\dR^{\leq 0,>n} \\
		\Ec & \mapsto \Ec_+^n
	\end{split}
\]
which is a left adjoint to the forgetful functor from \(\Bun_{\XFF}(G_K)_\dR^{\leq 0,>n}\) to \(\Bun_{\XFF}(G_K)_\dR\), and which associates with a vector bundle \(\Ec = (\Ec_\e, \Ec_\dR^+, \iota_{\Ec})\) the vector bundle
\[
	\Ec_+^n = (\Ec_\e, \Ec_\dR^+ + \BdR^+ \otimes_K \Fil^{-n} \DDb_\dR(\Ec), \iota_{\Ec}).
\]

The construction of the functor \(\tau_\dR^{\leq 0,>n}\) is similar to the one of \(\tau_\dR^{\leq 0}\).
Therefore, we will omit some details and refer the reader to~\cite{Ponsinet2024:1}.
In particular, this section~\ref{sec:truncation} purposefully follows the same structure as~\cite[\S 2]{Ponsinet2024:1} for the reader's convenience.

\subsection{Modification of filtered vector spaces}
\label{subsec:modif_Fil}

Let \(n \geq 1\) be an integer.
Let \(\Fil_K^{\leq 0,>n}\) be the full subcategory of \(\Fil_K\) of filtered \(K\)\=/vector spaces whose weights are all in the set \(\Z \setminus [1,n]\).

We define the functor
\[
	\begin{split}
		\tau_{\Fil}^{\leq 0,>n} : \Fil_K & \rightarrow \Fil_K^{\leq 0,>n} \\
		(D,\Fil D) & \mapsto (D,\Fil_{+,n} D)
	\end{split}
\]
which associates with a filtered \(K\)\=/vector space \((D,\Fil D)\) the filtered  \(K\)\=/vector space \((D,\Fil_{+,n} D)\) where
\[
	\Fil_{+,n}^i D = \begin{cases}
		\Fil^i D, & \text{if } i \leq -n, \\
		\Fil^{-n} D,  & \text{if } i \in [-n,0], \\
		\Fil^i D, & \text{if } i > 0.
	\end{cases}
\]
Note that the identity map on \(D\) induces a morphism of filtered \(K\)\=/vector spaces \(\eta_D : (D,\Fil D) \rightarrow (D,\Fil_{+,n} D)\).

\begin{proposition} \label{proposition:trunc_fil}
	The functor \(\tau_{\Fil}^{\leq 0,>n}\) is exact and left adjoint to the forgetful functor from \(\Fil_K^{\leq 0,>n}\) to \(\Fil_K\).
	Moreover, we have the following properties.
	\begin{enumerate}
		\item Let \((D,\Fil D)\) be a filtered \(K\)\=/vector space.
		The morphism \(\eta_D\) is the universal morphism from \((D,\Fil D)\) to the forgetful functor from \(\Fil_K^{\leq 0,>n}\) to \(\Fil_K\).
		\item There is a commutative diagram
		\[
			\begin{tikzcd}
				\Fil_K \ar{r}{\tau_{\Fil}^{\leq 0,>n}} \ar{d} & \Fil_K^{\leq 0,>n} \ar{d} \\
				\Mod_K \ar{r}{\id} & \Mod_K,
			\end{tikzcd}
		\]
		where the vertical arrows are the forgetful functor \((D,\Fil D) \mapsto D\), and the bottom arrow is the identity functor.
	\end{enumerate}
\end{proposition}
\begin{proof}
	The statement is proved similarly to~\cite[Proposition~2.1.4, Remark~2.1.3, and Corollary 2.1.5]{Ponsinet2024:1}.
\end{proof}

\begin{lemma} \label{lemma:module_modification_filtration}
	Let \((D,\Fil D)\) be a filtered \(K\)\=/vector space.
	Then
	\[
		\sum_{i \in \Z} \Fil^i \BdR^+ \otimes_K \Fil_{+,n}^{-i} D = \left(\sum_{i \in \Z} \Fil^i \BdR^+ \otimes_K \Fil^{-i} D \right) + \BdR^+ \otimes_K \Fil^{-n} D.
	\]
\end{lemma}
\begin{proof}
	The statement follows from computations similar to~\cite[Lemma~2.1.6]{Ponsinet2024:1}.
	By definition, we have
	\[
		\begin{split}
			\sum_{i \in \Z} \Fil^i \BdR^+ \otimes_K \Fil_{+,n}^{-i} D = & \sum_{i < 0} \Fil^i \BdR^+ \otimes_K \Fil^{-i} D \\
			& + \sum_{i \in [0,n]} \Fil^i \BdR^+ \otimes_K \Fil^{-n} D \\
			& + \sum_{i \geq n} \Fil^i \BdR^+ \otimes_K \Fil^{-i} D
		\end{split}
	\]
	On the one hand, since \(\Fil^i \BdR^+ \otimes_K \Fil^{-i} D \subseteq \Fil^i \BdR^+ \otimes_K \Fil^{-n} D\) for each \(i \in [0,n]\), we have
	\[
		\begin{split}
			\sum_{i \in \Z} \Fil^i \BdR^+ \otimes_K \Fil_{+,n}^{-i} D & = \sum_{i \in \Z} \Fil^i \BdR^+ \otimes_K \Fil^{-i} D \\
			& + \sum_{i \in [0,n]} \Fil^i \BdR^+ \otimes_K \Fil^{-n} D.
		\end{split}
	\]
	On the other hand, we have
	\[
		\sum_{i \in [0,n]} \Fil^i \BdR^+ \otimes_K \Fil^{-n} D = \BdR^+ \otimes_K \Fil^{-n} D.
	\]
\end{proof}

Let \(\Rep_{\BdR^+}^\gfl(G_K)^{\leq 0,>n}\) be the subcategory of \(\Rep_{\BdR^+}^\gfl(G_K)\) equivalent to the subcategory \(\Fil^{\leq 0,>n}_K\) of \(\Fil_K\) via the equivalence \(\Rep_{\BdR^+}^\gfl(G_K) \similarrightarrow \Fil_K\) from Theorem~\ref{theorem:genflatFil}.
By Lemma~\ref{lemma:module_modification_filtration}, the composition of the functor \(\tau_{\Fil}^{\leq 0,>n}\) with the equivalences of categories \(\Rep_{\BdR^+}^\gfl(G_K) \similarrightarrow \Fil_K\) and \(\Rep_{\BdR^+}^\gfl(G_K)^{\leq 0,>n} \similarrightarrow \Fil_K^{\leq 0,>n}\) then yields a functor
\[
	\begin{split}
		\tau_{\dR}^{\leq 0,>n} : \Rep_{\BdR^+}^\gfl(G_K) & \rightarrow \Rep_{\BdR^+}^\gfl(G_K)^{\leq 0,>n} \\
		M_\dR^+ & \mapsto M_\dR^+ + \BdR^+ \otimes_K (t^{-n} M_\dR^+)^{G_K}.
	\end{split}
\]

Proposition~\ref{proposition:trunc_fil} implies the following.

\begin{proposition} \label{proposition:trunc_BdR+}
	The functor \(\tau_{\dR}^{\leq 0,>n}\) is exact and left adjoint to the forgetful functor from \(\Rep_{\BdR^+}^\gfl(G_K)^{\leq 0,>n}\) to \(\Rep_{\BdR^+}^\gfl(G_K)\).
	Moreover, we have the following properties.
	\begin{enumerate}
		\item Let \(M_\dR^+\) be a generically flat \(\BdR^+\)\=/representation of \(G_K\).
		The inclusion morphism \(M_\dR^+ \subset M_\dR^+ + \BdR^+ \otimes_K (t^{-n} M_\dR^+)^{G_K}\) is the universal morphism from \(M_\dR^+\) to the forgetful functor from \(\Rep_{\BdR^+}^\gfl(G_K)^{\leq 0,>n}\) to \(\Rep_{\BdR^+}^\gfl(G_K)\).
		\item There is a commutative diagram
		\[
			\begin{tikzcd}
				\Rep_{\BdR^+}^\gfl(G_K) \ar{r}{\tau_{\dR}^{\leq 0,>n}} \ar{d} & \Rep_{\BdR^+}^\gfl(G_K)^{\leq 0,>n} \ar{d} \\
				\Rep_{\BdR}^\fl(G_K) \ar{r}{\id} & \Rep_{\BdR}^\fl(G_K),
			\end{tikzcd}
		\]
		where the vertical arrows are the functor of extension of scalars, and the bottom arrow is the identity functor.
	\end{enumerate}
\end{proposition}

\subsection{Modification of de Rham vector bundles}
\label{subsec:modif_BundR}

Let \(n \geq 1\) be an integer.
Let \(\Bun_{\XFF}(G_K)_\dR^{\leq 0,>n}\) be the full subcategory of \(\Bun_{\XFF}(G_K)_\dR\) of \(G_K\)\=/equivariant de Rham vector bundles over \(\XFF\) whose Hodge\--Tate weights are all in the set \(\Z \setminus [1,n]\).

By Remark~\ref{remark:Bun_dR} and Proposition~\ref{proposition:trunc_BdR+}, the functor \(\tau_{\dR}^{\leq 0,>n}\) together with the identity functor on \(\Rep_{\Be}(G_K)\) induces a functor
\[
	\begin{split}
		\tau_{\HT}^{\leq 0,>n} : \Bun_{\XFF}(G_K)_\dR & \rightarrow \Bun_{\XFF}(G_K)_\dR^{\leq 0,>n} \\
		\Ec & \mapsto \Ec_+^n
	\end{split}
\]
which associates with a de Rham vector bundle \(\Ec = (\Ec_\e,\Ec_\dR^+,\iota_{\Ec})\) the de Rham vector bundle
\[
	\Ec_+^n = (\Ec_\e,\Ec_\dR^+ + \BdR^+ \otimes_K \Fil^{-n} \DDb_\dR(\Ec),\iota_{\Ec}).
\]

\begin{proposition} \label{proposition:truncation_HT}
	The functor \(\tau_{\HT}^{\leq 0,>n}\) is exact and left adjoint to the forgetful functor from \(\Bun_{\XFF}(G_K)_\dR^{\leq 0,>n}\) to \(\Bun_{\XFF}(G_K)_\dR\).
	Moreover, we have the following properties.
	\begin{enumerate}
		\item Let \(\Ec\) be a \(G_K\)\=/equivariant de Rham vector bundle.
		The inclusion map \(\Ec \rightarrow \Ec_+^n\) is the universal morphism from \(\Ec\) to the forgetful functor from \(\Bun_{\XFF}(G_K)_\dR^{\leq 0,>n}\) to \(\Bun_{\XFF}(G_K)_\dR\).
		\item The composition of the functor \(\tau_{\HT}^{\leq 0,>n}\) with the forgetful functor
		\[
			\Bun_{\XFF}(G_K)_\dR^{\leq 0,>n} \rightarrow \Bun_{\XFF}(G_K)_\dR^{\leq 0,>n} \xrightarrow{\tau_{\HT}^{\leq 0,>n}} \Bun_{\XFF}(G_K)_\dR^{\leq 0,>n}
		\]
		is isomorphic to the identity functor on \(\Bun_{\XFF}(G_K)_\dR^{\leq 0,>n}\).
	\end{enumerate}
\end{proposition}
\begin{proof}
	All statements except the last follow from Proposition~\ref{proposition:trunc_BdR+}.
	The last statement is a formal consequence of the adjunction property of \(\tau_{\HT}^{\leq 0,>n}\) and the fully faithfulness of the forgetful functor (see \cite[IV \S 3 Theorem~1]{MacLane1998}).
\end{proof}

\subsection{Hodge\--Tate and Harder--Narasimhan filtrations}
\label{subsec:HT_HN}

Let \(n \geq 1\) be an integer.
Let \(V\) be a de Rham \(p\)\=/adic representation of \(G_K\).
The vector bundle \(\Ec(V)\) associated with \(V\) is \(G_K\)\=/equivariant and de Rham, and we denote its modification by \(\tau_\dR^{\leq 0,>n}\) by
\[
	\Ec_+^n(V) = (\Be \otimes_{\Qp} V, \BdR^+ \otimes_{\Qp} V + \BdR^+ \otimes_K \Fil^{-n} \DDb_\dR(V), \iota_{\Ec(V)}).
\]
The inclusion map of \(\Ec(V)\) in \(\Ec_+^n(V)\) induces a short exact sequence in \(\Coh_{\XFF}(G_K)\)
\begin{equation} \label{eq:EV}
	0 \rightarrow \Ec(V) \rightarrow \Ec_+^n(V) \rightarrow \Fc_+^n(V) \rightarrow 0,
\end{equation}
with
\[
	\Fc_+^n(V)=\left(0,\frac{\BdR^+ \otimes_{\Qp} V + \BdR^+ \otimes_K \Fil^{-n} \DDb_\dR(V)}{\BdR^+ \otimes_{\Qp} V}\right).
\]

\begin{lemma} \label{lemma:HN_geq0}
	The Harder\--Narasimhan slopes of \(\Ec^n_+(V)\) are all \(\geq 0\).
\end{lemma}
\begin{proof}
	The short exact sequence~\eqref{eq:EV} induces the cohomological exact sequence
	\begin{equation} \label{eq:cohomHN0}
		\HH^1(\XFF,\Ec(V)) \rightarrow \HH^1(\XFF,\Ec^n_+(V)) \rightarrow \HH^1(\XFF,\Fc^n_+(V)) \rightarrow 0.
	\end{equation}
	Since \(\Ec(V)\) and \(\Fc^n_+(V)\) are semi-stable of slopes \(0\) and \(+\infty\) respectively, the groups \(\HH^1(\XFF,\Ec(V))\) and \(\HH^1(\XFF,\Fc^n_+(V))\) are trivial by Proposition~\ref{proposition:cohomology_HN}.
	Therefore, the exact sequence~\eqref{eq:cohomHN0} forces the group \(\HH^1(\XFF,\Ec^n_+(V))\) to vanish, and the lemma follows from another application of Proposition~\ref{proposition:cohomology_HN}.
\end{proof}

\begin{notation}
	We denote by
	\[
		0 \rightarrow V_{0,n} \rightarrow V \rightarrow V^{\leq 0,>n} \rightarrow 0
	\]
	the short exact sequence of \(p\)\=/adic representations of \(G_K\) where \(V^{\leq 0,> n}\) is the maximal quotient representation of \(V\) whose Hodge\--Tate weights are all in the set \(\Z \setminus [1,n]\).
\end{notation}

\begin{lemma} \label{lemma:V0n}
	There exists no non-trivial quotient of \(V_{0,n}\) whose Hodge\--Tate weights all in \(\Z \setminus [1,n]\).
\end{lemma}
\begin{proof}
	Let \(U\) be a subrepresentation of \(V_{0,n}\) such that the Hodge\--Tate weights of \(V_{0,n}/U\) are all in \(\Z \setminus [1,n]\).
	Then there is a short exact sequence of \(p\)\=/adic representation of \(G_K\)
	\[
		0 \rightarrow V_{0,n}/U \rightarrow V/U \rightarrow V^{\leq 0,>n} \rightarrow 0.
	\]
	By Proposition~\ref{proposition:dR}, the representation \(V/U\) and its Hodge\--Tate weights are all in \(\Z \setminus [1,n]\).
	By maximality of \(V^{\leq 0,>n}\), we then have \(U=V_{0,n}\).
\end{proof}

By Theorem~\ref{theorem:almost_coh}, the short exact sequence of de Rham representations
\[
	0 \rightarrow V_{0,n} \rightarrow V \rightarrow V^{\leq 0,>n} \rightarrow 0
\]
induces a short exact sequence in \(\Bun_{\XFF}(G_K)_\dR\)
\[
	0 \rightarrow \Ec(V_{0,n}) \rightarrow \Ec(V) \rightarrow \Ec(V^{\leq 0,>n}) \rightarrow 0,
\]
which in turns, by exactness of \(\tau_\HT^{\leq 0,>n}\) from Proposition~\ref{proposition:truncation_HT}, induces a short exact sequence in \(\Bun_{\XFF}(G_K)_\dR^{\leq 0,>n}\)
\[
	0 \rightarrow \Ec_+^n(V_{0,n}) \rightarrow \Ec_+^n(V) \rightarrow \Ec_+^n(V^{\leq 0,>n}) \rightarrow 0.
\]

\begin{proposition} \label{proposition:HN_E+n}
	The short exact sequence
	\[
		0 \rightarrow \Ec_+^n(V_{0,n}) \rightarrow \Ec_+^n(V) \rightarrow \Ec_+^n(V^{\leq 0,>n}) \rightarrow 0
	\]
	is the first step of the Harder\--Narasimhan filtration of \(\Ec_+^n(V)\) with
	\[
		\Ec_+^n(V^{\leq 0,>n}) = \Ec(V^{\leq 0,>n})
	\]
	semistable of slope \(0\).
\end{proposition}
\begin{proof}
	The statement is proved similarly to~\cite[Proposition~2.3.5 and Corollary~2.3.6]{Ponsinet2024:1}.
	By Lemma~\ref{lemma:HN_geq0}, the Harder\--Narasimhan slopes of each vector bundle in the short exact sequences are all \(\geq 0\).
	Moreover, since the Hodge\--Tate weights of \(V^{\leq 0,>n}\) are all in \(\Z \setminus [1,n]\), by Proposition~\ref{proposition:truncation_HT}, we have \(\Ec_+^n(V^{\leq 0,>n}) = \Ec(V^{\leq 0,>n})\) which is semistable of slope \(0\).
	By uniqueness of the Harder\--Narasimhan filtration, it remains to prove that the Harder\--Narasimhan slopes of \(\Ec_+^n(V_{0,n})\) are all \(>0\).

	Assume that \(0\) is a Harder\--Narasimhan slope of \(\Ec_+^n(V_{0,n})\), and let \(\Hc\) be the first step of the Harder\--Narasimhan filtration of \(\Ec_+^n(V_{0,n})\), that is, the vector bundle \(\Hc\) is semistable of slope \(0\) and there exists a surjective map
	\[
		f: \Ec_+^n(V_{0,n}) \rightarrow \Hc \rightarrow 0.
	\]
	By Remark~\ref{remark:HN_equiv}, the vector bundle \(\Hc\) and the map \(f\) are \(G_K\)\=/equivariant.
	Since \(\Hc\) is semistable of slope \(0\), there exists a \(p\)\=/adic representation \(W\) of \(G_K\) such that \(\Hc \similarrightarrow \Ec(W)\) by Theorem~\ref{theorem:almost_coh}.
	Moreover, by Proposition~\ref{proposition:dR}, the surjection \(f\) implies that \(\Hc\), and thus \(W\), is de Rham with Hodge\--Tate weights all in \(\Z \setminus [1,n]\).
	Therefore, we have
	\[
		\begin{split}
			f \in  \Hom_{\Bun_{\XFF}(G_K)_\dR^{\leq 0,>n}}(\Ec_+^n(V_{0,n}),\Hc) & \similarrightarrow \Hom_{\Bun_{\XFF}(G_K)_\dR}(\Ec(V_{0,n}),\Hc) \\
			& \similarrightarrow \Hom_{\Rep_{\Qp}(G_K)}(V_{0,n},W),
		\end{split}
	\]
	where the first isomorphism follows from the adjunction from Proposition~\ref{proposition:truncation_HT}, the second from the equivalence from Theorem~\ref{theorem:almost_coh}.

	By Lemma~\ref{lemma:V0n} and Proposition~\ref{proposition:dR}, the group \(\Hom_{\Rep_{\Qp}(G_K)}(V_{0,n},W)\) is trivial.
	In particular, the map \(f\) is trivial, which contradicts the assumption that \(0\) is a Harder\--Narasimhan slope of \(\Ec_+^n(V_{0,n})\).
\end{proof}

\section{Bloch--Kato groups over perfectoid fields}
\label{sec:BK}

\subsection{The Bloch--Kato groups}
\label{subsec:BK}

We recall the definition of the Bloch\--Kato groups~\cite[\S 3]{BlochKato1990}, and we define a filtration on the exponential Bloch\--Kato groups induced by the filtration of \(\BdR\).

Let \(V\) be a \(p\)\=/adic representation of \(G_K\).
For each finite extension \(K^\prime\) of \(K\), the \emph{exponential}, \emph{finite} and \emph{geometric} \emph{Bloch\--Kato groups} are respectively defined by
\[
	\begin{split}
			\HH^1_e(K^\prime,V) = & \Ker\left(\HH^1(K^\prime,V) \rightarrow \HH^1(K^\prime,\Be \otimes_{\Qp} V) \right), \\
			\HH^1_f(K^\prime,V) = & \Ker\left(\HH^1(K^\prime,V) \rightarrow \HH^1(K^\prime,\Bcris \otimes_{\Qp} V) \right), \\
			\HH^1_g(K^\prime,V) = & \Ker\left(\HH^1(K^\prime,V) \rightarrow \HH^1(K^\prime,\BdR \otimes_{\Qp} V) \right).
	\end{split}
\]

Recall that the filtration of \(\BdR\) induces a filtration on \(\Be\).
We define a filtration on the exponential Bloch\--Kato groups as follows.

\begin{definition}
	For each \(n \in \Z\), we set
	\[
		\Fil^n \HH^1_e(K^\prime,V) =
		\begin{cases}
			\Ker\left(\HH^1(K^\prime,V) \rightarrow \HH^1(K^\prime,\Fil^n \Be \otimes_{\Qp} V) \right), & \text{if } n \leq 0, \\
			0, & \text{if } n \geq 0.
		\end{cases}
	\]
\end{definition}

Let \(T\) be a \(G_K\)\=/stable lattice in \(V\).
The short exact sequence of topological \(G_K\)\=/modules
\[
	0 \rightarrow T \rightarrow V \rightarrow V/T \rightarrow 0
\]
induces an exact sequence
\[
	\HH^1(K^\prime,T) \xrightarrow{\alpha} \HH^1(K^\prime,V) \xrightarrow{\beta} \HH^1(K^\prime,V/T).
\]
For \(\ast \in \{e,f,g\}\), the Bloch\--Kato subgroups of \(\HH^1(K^\prime,T)\) and \(\HH^1(K^\prime,V/T)\) are respectively defined by
\[
	\begin{split}
		\HH^1_\ast(K^\prime,T) = & \alpha^{-1}\left(\HH^1_\ast(K^\prime,V)\right), \\
		\HH^1_\ast(K^\prime,V/T) = & \beta\left(\HH^1_\ast(K^\prime,V)\right).
	\end{split}
\]
Moreover, the exponential Bloch\--Kato groups are equipped with the induced filtrations
\[
	\begin{split}
		\Fil^n \HH^1_e(K^\prime,T) = & \alpha^{-1}\left(\Fil^n \HH^1_e(K^\prime,V)\right), \\
		\Fil^n \HH^1_e(K^\prime,V/T) = & \beta\left(\Fil^n \HH^1_e(K^\prime,V)\right).
	\end{split}
\]

Let \(L\) be an algebraic extension of \(K\).
Recall that there is a natural isomorphism
\[
	\HH^1(L,V/T) \similarrightarrow \varinjlim_{\res, K^\prime} \HH^1(K^\prime,V/T),
\]
where \(K^\prime\) runs over all the finite extensions of \(K\) contained in \(L\), and the transition morphisms are the restriction maps~\cite[I \S 2.2 Proposition~8]{Serre1994}.
For each \(\ast \in \{e,f,g\}\), the groups \(\HH^1_\ast(K^\prime,V/T)\) and the filtration \(\Fil^n \HH^1_e(K^\prime,V/T)\) are compatible under the restriction maps, and the Bloch\--Kato subgroups of \(\HH^1(L,V/T)\) are then defined by
\[
	\begin{split}
		\HH^1_\ast(L,V/T) = & \varinjlim_{\res, K^\prime} \HH^1_\ast(K^\prime, V/T), \\
		\Fil^n \HH^1_e(L,V/T) = & \varinjlim_{\res, K^\prime} \Fil^n \HH^1_e(K^\prime, V/T),
	\end{split}
\]
where \(K^\prime\) runs over all the finite extensions of \(K\) contained in \(L\), and the transition morphisms are the restriction maps.

\subsection{Universal norms}
\label{subsec:universal_norms}

Let \(V\) be a \(p\)\=/adic representation of \(G_K\).
Let \(T\) be a \(G_K\)\=/stable lattice in \(V\).
Let \(L\) be an algebraic extension of \(K\).

For \(i \in \N\), the \emph{\(i\)\=/th Iwasawa cohomology group} \(\HH^i_\Iw(K,L,T)\) of the extension \(L/K\) with coefficients in \(T\) is defined by
\[
	\HH^i_\Iw(K,L,T) = \varprojlim_{\cores, K^\prime} \HH^i(K^\prime,T),
\]
where \(K^\prime\) runs over all the finite extensions of \(K\) contained in \(L\), and the transition morphisms are the corestriction maps.

For each \(\ast \in \{e,f,g\}\), the Bloch\--Kato groups \(\HH^1_\ast(K^\prime,T)\) are compatible under the corestriction maps.
The modules of \emph{\(\ast\)\=/universal norms} associated with \(T\) in the extension \(L/K\) are defined by
\[
		\HH^1_{\Iw,\ast}(K,L,T) = \varprojlim_{\cores, K^\prime} \HH^1_\ast(K^\prime, T),
\]
where \(K^\prime\) runs over all the finite extensions of \(K\) contained in \(L\), and the transition morphisms are the corestriction maps.

Let \(V^\ast(1) = \Hom_{\Qp}(V,\Qp(1))\) be the Tate dual representation of \(V\), and let \(T^\ast(1) = \Hom_{\Zp}(T,\Zp(1))\) be the \(G_K\)\=/stable lattice in \(V^\ast(1)\) Tate dual of \(T\).
Recall that local Tate duality~\cite{Serre1994}, for each finite extension \(K^\prime\) of \(K\),
\[
	\HH^1(K^\prime,V^\ast(1)/T^\ast(1)) \times \HH^1(K^\prime,T) \rightarrow \HH^2(K^\prime,\Qp(1)/\Zp(1)) \cong \Qp/\Zp
\]
induces a perfect pairing
\[
	\HH^1(L,V^\ast(1)/T^\ast(1)) \times \HH^1_\Iw(K,L,T) \rightarrow \Qp/\Zp.
\]

Bloch and Kato~\cite[Proposition~3.8]{BlochKato1990} have proved the following duality properties for the Bloch\--Kato groups.

\begin{proposition}[Bloch\--Kato] \label{proposition:BK_duality}
	If \(V\) is de Rham, then, under local Tate duality,
	\begin{enumerate}
		\item the orthogonal complement of \(\HH^1_e(L,V^\ast(1)/T^\ast(1))\) is \(\HH^1_{\Iw,g}(K,L,T)\),
		\item the orthogonal complement of \(\HH^1_f(L,V^\ast(1)/T^\ast(1))\) is \(\HH^1_{\Iw,f}(K,L,T)\),
		\item the orthogonal complement of \(\HH^1_g(L,V^\ast(1)/T^\ast(1))\) is \(\HH^1_{\Iw,e}(K,L,T)\).
	\end{enumerate}
\end{proposition}

\subsection{Comparison of the Bloch--Kato groups}
\label{subsec:comparison}

Let \(V\) be a \(p\)\=/adic representation of \(G_K\), and let \(T\) be a \(G_K\)\=/stable lattice in \(V\).
For each finite extension \(K^\prime\) of \(K\), let
\[
	\DDb_{\cris,K^\prime}(V) = \HH^0(K^\prime, \Bcris \otimes_{\Qp} V).
\]
Recall that \(\DDb_{\cris,K^\prime}(V)\) is a finite dimensional \(K^\prime_0\)\=/vector space equipped with a map \(\varphi\) semilinear with respect to the absolute Frobenius on \(K^\prime_0\) (see~\cite[\S 5.1]{Fontaine1994:III}).
For \(i \in \Z\), we set the finite dimensional \(\Qp\)\=/vector space
\[
	\DDb_{\cris,K^\prime}(V)^{\varphi=p^i} = \{x \in \DDb_{\cris,K^\prime}(V), \varphi(x)=p^i\cdot x\}.
\]

\begin{proposition} \label{proposition:dim_phi_invariant}
	Let \(i \in \Z\).
	The dimension of the \(\Qp\)\=/vector space \(\DDb_{\cris,K^\prime}(V)^{\varphi=p^i}\) is bounded independently of \(K^\prime\).
\end{proposition}
\begin{proof}
	Let
	\[
		\DDb_{\pcris}(V) = \varinjlim_{\res,K^\prime} \HH^0(K^\prime,\Bcris \otimes_{\Qp} V),
	\]
	where \(K^\prime\) runs over all the finite extensions of \(K\), and the transition morphisms are the restriction maps.
	Then \(\DDb_{\pcris}(V)\) is a finite dimensional discrete \((\varphi,G_K)\)\=/module over \(\Qp^{\ur}\), that is, \(\DDb_{\pcris}(V)\) is a finite dimensional \(\Qp^{\ur}\)\=/vector space equipped with a map \(\varphi\) semilinear with respect to the absolute Frobenius on \(\Qp^{\ur}\) and a discrete action of \(G_K\) commuting with \(\varphi\) (see~\cite[\S 5.6]{Fontaine1994:III}).
	We set
	\[
		\widehat{\DDb}_{\pcris}(V) = \hat{\Q}_p^{\ur} \otimes_{\Qp^\ur} \DDb_{\pcris}(V).
	\]
	Then (see~\cite[Remarque 4.4.10]{Fontaine1994:III}), \(\widehat{\DDb}_{\pcris}(V)\) is a finite dimensional \((\varphi,G_K)\)\=/module over \(\hat{\Q}_p^{\ur}\), and for each finite extension \(K^\prime\) of \(K\), we have
	\[
		\DDb_{\cris,K^\prime}(V) = \DDb_{\pcris}(V)^{G_{K^\prime}} = \widehat{\DDb}_{\pcris}(V)^{G_{K^\prime}}.
	\]
	Hence, since the action of \(G_K\) and \(\varphi\) commute, we have
	\[
		\DDb_{\cris,K^\prime}(V)^{\varphi=p^i} = (\widehat{\DDb}_{\pcris}(V)^{\varphi=p^i})^{G_{K^\prime}}.
	\]
	By the Dieudonné\--Manin theorem~\cite[IV \S 4]{Demazure1972}, there is an isomorphism of \(\varphi\)\=/modules over \(\hat{\Q}_p^\ur\)
	\[
		\widehat{\DDb}_{\pcris}(V) \similarrightarrow \bigoplus_{v \in \Q} E(v)^{\oplus m_v},
	\]
	where \(E(v)\) runs over the simple objects in the category of the \(\varphi\)\=/modules over \(\hat{\Q}_p^\ur\), that is, if \(v=s/r\) with \(s,r \in \N\), \(r>0\), and \((r,s)=1\), then \(E(v) = \hat{\Q}_p^\ur \otimes_{\Qp} \Qp[T]/(T^r - p^s)\) on which \(\varphi\) acts by multiplication by \(T\) semilinear with respect to the absolute Frobenius on \(\hat{\Q}_p^\ur\).
	We then have
	\[
		\widehat{\DDb}_{\pcris}(V)^{\varphi=p^i} \similarrightarrow (E(i)^{\varphi=p^i})^{\oplus m_i},
	\]
	and \(E(i)^{\varphi=p^i}\) is a finite dimensional \(\Qp\)\=/vector space~\cite[IV \S 2 and \S 3]{Demazure1972}.
\end{proof}

\begin{proposition}[Bloch\--Kato] \label{proposition:BK_quotient_dim}
	Let \(K^\prime\) be a finite extension of \(K\).
	If \(V\) is de Rham, then we have
	\[
		\begin{split}
			\dim_{\Qp} \HH^1_f(K^\prime,V)/\HH^1_e(K^\prime,V) & = \dim_{\Qp}  \DDb_{\cris,K^\prime}(V)^{\varphi = 1}, \\
			\dim_{\Qp} \HH^1_g(K^\prime,V)/\HH^1_f(K^\prime,V) & = \dim_{\Qp} \DDb_{\cris,K^\prime}(V)^{\varphi = p^{-1}}.
		\end{split}
	\]
\end{proposition}
\begin{proof}
	Bloch and Kato~\cite[Corollary~3.8.4]{BlochKato1990} have proved that there exists an isomorphism of \(\Qp\)\=/vector spaces
	\[
		\HH^1_f(K^\prime,V)/\HH^1_e(K^\prime,V) \similarrightarrow \DDb_{\cris,K^\prime}(V)/(1 - \varphi)\DDb_{\cris,K^\prime}(V),
	\]
	which implies the first statement.
	By the duality from Proposition~\ref{proposition:BK_duality}, the first statement yields
	\[
		\begin{split}
			\dim_{\Qp} \HH^1_g(K^\prime,V)/\HH^1_f(K^\prime,V) & = \dim_{\Qp} \HH^1_f(K^\prime,V^\ast(1))/\HH^1_e(K^\prime,V^\ast(1)) \\
			& = \dim_{\Qp} \DDb_{\cris,K^\prime}(V^\ast(1))^{\varphi = 1},
		\end{split}
	\]
	and the duality of \(\varphi\)\=/modules~\cite[\S 5.1]{Fontaine1994:III}
	\[
		\DDb_{\cris,K^\prime}(V) \otimes_{K^\prime_0} \DDb_{\cris,K^\prime}(V^\ast(1)) \similarrightarrow \DDb_{\cris,K^\prime}(\Qp(1)),
	\]
	implies the equality
	\[
		\dim_{\Qp} \DDb_{\cris,K^\prime}(V^\ast(1))^{\varphi = 1} = \dim_{\Qp}  \DDb_{\cris,K^\prime}(V)^{\varphi = p^{-1}}.
	\]
\end{proof}

\begin{proposition} \label{proposition:comparison}
	Let \(L\) be an algebraic extension of \(K\).
	If \(V\) is de Rham, then the Pontryagin dual of the quotient
	\[
		\HH^1_g(L,V/T)/\HH^1_e(L,V/T)
	\]
	is a free \(\Zp\)\=/module of finite rank bounded independently of \(L\).
\end{proposition}
\begin{proof}
	By Proposition~\ref{proposition:BK_duality}, the Pontryagin dual of the discrete \(\Zp\)\=/module
	\[
		\HH^1_g(L,V/T)/\HH^1_e(L,V/T)
	\]
	is the compact and Hausdorff \(\Zp\)\=/module
	\[
		\HH^1_{\Iw,g}(K,L,T^\ast(1))/\HH^1_{\Iw,e}(K,L,T^\ast(1)).
	\]
	By definition, there exists an injective map
	\begin{equation} \label{eq:BK_comparison_limit}
		0 \rightarrow \HH^1_{\Iw,g}(K,L,T^\ast(1))/\HH^1_{\Iw,e}(K,L,T^\ast(1)) \rightarrow  \varprojlim_{\cores, K^\prime} \HH^1_g(K^\prime,V^\ast(1))/\HH^1_e(K^\prime,V^\ast(1)),
	\end{equation}
	where \(K^\prime\) runs over all the finite extensions of \(K\) contained in \(L\), and the transition morphisms are the corestriction maps.

	By Proposition~\ref{proposition:dim_phi_invariant} and Proposition~\ref{proposition:BK_quotient_dim}, the dimension of the \(\Qp\)\=/vector space \(\HH^1_g(K^\prime,V^\ast(1))/\HH^1_e(K^\prime,V^\ast(1))\) is bounded independently of \(K^\prime\).
	Therefore, the \(\Qp\)\=/vector space \(\varprojlim \HH^1_g(K^\prime,V^\ast(1))/\HH^1_e(K^\prime,V^\ast(1))\) is finite dimensional, and we conclude using the map~\eqref{eq:BK_comparison_limit}.
\end{proof}

\subsection{Universal extensions and groups of points}
\label{subsec:universal_ext_grps_pts}

We recall the definition and properties of universal objects in \(\Cc(G_K)\) and of the groups of points both associated with a \(p\)\=/adic representation by Fontaine~\cite[\S 8]{Fontaine2003}.

Let \(V\) be a \(p\)\=/adic representation of \(G_K\).
The tangent space \(t_V\) associated with \(V\) is the \(K\)\=/vector space
\[
	t_V = ((\BdR/\BdR^+) \otimes_{\Qp} V)^{G_K},
\]
which is equipped with the filtration by \(K\)\=/vector subspaces
\[
	\Fil^n t_V =
	\begin{cases}
		((\Fil^n \BdR/\BdR^+) \otimes_{\Qp} V)^{G_K}, & \text{if } n \leq 0, \\
		0, & \text{if } n > 0.
	\end{cases}
\]
If \(V\) is de Rham, then they are isomorphisms
\[
	\DDb_\dR(V)/\Fil^0 \DDb_\dR(V) \similarrightarrow t_V,
\]
and, for each \(n \in \N\),
\[
	\Fil^{-n} \DDb_\dR(V)/\Fil^0 \DDb_\dR(V) \similarrightarrow \Fil^{-n} t_V.
\]

We set
\[
	\begin{split}
		t_V(\Qpbar) & = ((\BdR/\BdR^+) \otimes_{\Qp} V)_\delta, \\
		\Fil^{-n} t_V(\Qpbar) & = ((\Fil^{-n} \BdR/\BdR^+) \otimes_{\Qp} V)_\delta.
	\end{split}
\]
Let \(\hat{t}_V(\Qpbar)\) be the topological closure of the image of \(t_V(\Qpbar)\) in \({(\BdR/\BdR^+) \otimes_{\Qp} V}\), and, for each \(n \in \N\), let \(\Fil^{-n} \hat{t}_V(\Qpbar)\) be the topological closure of the image of \(\Fil^{-n} t_V(\Qpbar)\) in \({(\Fil^{-n} \BdR/\BdR^+) \otimes_{\Qp} V}\).

Let \(t_V(\BdR^+) = t_V \otimes_K \BdR^+\), and, for each \(n \in \N\), let \(\Fil^{-n} t_V(\BdR^+) = \Fil^{-n} t_V \otimes_K \BdR^+\).
Note that there are natural morphisms of \(\BdR^+\)\=/modules by extension of scalars
\[
	t_V(\BdR^+) \rightarrow (\BdR/\BdR^+) \otimes_{\Qp} V
\]
and, for each \(n \in \N\),
\[
	\Fil^{-n} t_V(\BdR^+) \rightarrow (\Fil^{-n} \BdR/\BdR^+) \otimes_{\Qp} V.
\]

A \emph{trivial torsion \(\BdR^+\)\=/representation of \(G_K\)} (respectively a \emph{trivial \(\Bb_n\)\=/representation of \(G_K\)}) is a torsion \(\BdR^+\)\=/representation of \(G_K\) isomorphic to \(\bigoplus_{i \in \N} \Bb_i^{\oplus m_i}\) (respectively \(\bigoplus_{i \in [1,n]} \Bb_i^{\oplus m_i}\)), for some integers \(m_i\).

\begin{proposition} \label{proposition:t_V}
	The modules associated with the tangent space of \(V\) satisfy the following properties.
	\begin{enumerate}
		\item
		\begin{enumerate}
			\item There is an isomorphism of discrete \(G_K\)\=/modules
			\[
				t_V(\Qpbar) \similarrightarrow t_V \otimes_K \Qpbar
			\]
			\item The module \(\hat{t}_V(\Qpbar)\) is the maximal trivial torsion \(\BdR^+\)\=/subrepresentation of \((\BdR/\BdR^+) \otimes_{\Qp} V\).
			\item There is an isomorphism of torsion \(\BdR^+\)\=/representations of \(G_K\)
			\[
				\hat{t}_V(\Qpbar) \similarrightarrow \Img\left(t_V(\BdR^+) \rightarrow (\BdR/\BdR^+) \otimes_{\Qp} V\right).
			\]
			\item If \(V\) is de Rham, then there exists an isomorphism of \(\BdR^+\)\=/representations of \(G_K\)
	\[
		\hat{t}_V(\Qpbar) \similarrightarrow \bigoplus_{i \in \N} \Bb_i^{\oplus m_i(V)},
	\]
	where \(m_i(V)\) is the multiplicity of \(i\) as a Hodge\--Tate weight of \(V\).
		\end{enumerate}
		\item
		\begin{enumerate}
			\item There is an isomorphism of discrete \(G_K\)\=/modules
			\[
				\Fil^{-n} t_V(\Qpbar) \similarrightarrow \Fil^{-n} t_V \otimes_K \Qpbar
			\]
			\item The module \(\Fil^{-n} \hat{t}_V(\Qpbar)\) is the maximal trivial torsion \(\Bb_n\)\=/subrepresentation of \((\Fil^{-n} \BdR/\BdR^+) \otimes_{\Qp} V\).
			\item There is an isomorphism of torsion \(\BdR^+\)\=/representations of \(G_K\)
			\[
				\Fil^{-n} \hat{t}_V(\Qpbar) \similarrightarrow \Img\left(\Fil^{-n} t_V(\BdR^+) \rightarrow (\Fil^{-n} \BdR/\BdR^+) \otimes_{\Qp} V\right).
			\]
			\item If \(V\) is de Rham, then there exists an isomorphism of \(\BdR^+\)\=/representations of \(G_K\)
			\[
				\Fil^{-n} \hat{t}_V(\Qpbar) \similarrightarrow \bigoplus_{i \in [1,n]} \Bb_i^{\oplus m_i(V)},
			\]
			where \(m_i(V)\) is the multiplicity of \(i\) as a Hodge\--Tate weight of \(V\).
		\end{enumerate}
	\end{enumerate}
\end{proposition}
\begin{proof}
	The first three points in both cases are due to Fontaine~\cite[Proposition~8.1]{Fontaine2003}.
	For last points under the assumption that \(V\) is de Rham, the statement for \(\hat{t}_V(\Qpbar)\) is \cite[Corollary~3.3.4]{Ponsinet2024:1}, while the statement for \(\Fil^n \hat{t}_V(\Qpbar)\) is proved similarly.
\end{proof}

We fix an integer \(n \geq 1\).
We set \(E_\e(V) = {\Be \otimes_{\Qp} V}\), and \(E_n(V) = {\Fil^{-n} \Be \otimes_{\Qp} V}\).
The tensor product of \(V\) with the fundamental exact sequences~\eqref{eq:funda} and~\eqref{eq:funda_fil} yields a commutative diagram with exact rows
\begin{equation} \label{eq:funda_V}
	\begin{tikzcd}
		0 \ar{r} & V \ar{r} & E_e(V) \ar{r} & (\BdR/\BdR^+) \otimes_{\Qp} V \ar{r} & 0 \\
		0 \ar{r} & V \ar{r} \ar[equal]{u} & E_n(V) \ar{r} \ar{u} & (\Fil^{-n} \BdR/\BdR^+) \otimes_{\Qp} V \ar{r} \ar{u} & 0.
	\end{tikzcd}
\end{equation}

Let \(E_+(V)\) be the reciprocal image of \(\hat{t}_V(\Qpbar)\) in \(E_\e(V)\), and let \(E_+^n(V)\) be the reciprocal image of \(\Fil^{-n} \hat{t}_V(\Qpbar)\) in \(E_n(V)\).
The diagram~\eqref{eq:funda_V} induces a commutative diagram with exact rows
\begin{equation} \label{eq:funda_+}
	\begin{tikzcd}
		0 \ar{r} & V \ar{r} & E_+(V) \ar{r} & \hat{t}_V(\Qpbar) \ar{r} & 0 \\
		0 \ar{r} & V \ar{r} \ar[equal]{u} & E^n_+(V) \ar{r} \ar{u} & \Fil^{-n} \hat{t}_V(\Qpbar) \ar{r} \ar{u} & 0.
	\end{tikzcd}
\end{equation}

Fontaine~\cite[\S 8.2]{Fontaine2003} has established the following properties of \(E_+(V)\) and \(E_+^n(V)\).

\begin{proposition}[Fontaine]
	The topological \(G_K\)\=/modules \(E_+(V)\) and \(E_+^n(V)\) are almost \(\Cp\)\=/representations of \(G_K\).
	Moreover, they satisfy the following universal properties.
	\begin{enumerate}
		\item The almost \(\Cp\)\=/representation \(E_+(V)\) is the universal extension of \(V\) by a trivial torsion \(\BdR^+\)\=/representation in \(\Bc(G_K)\).
		\item The almost \(\Cp\)\=/representation \(E_+^n(V)\) is the universal extension of \(V\) by a trivial \(\Bb_n\)\=/representation in \(\Bc(G_K)\).
	\end{enumerate}
\end{proposition}

\begin{lemma} \label{lemma:E+_Fil}
	Assume that \(V\) is de Rham.
	\begin{enumerate}
		\item If the Hodge\--Tate weights of \(V\) are all \(\leq 0\), then
		\[
			V = E_+^n(V) = E_+(V).
		\]
		\item If the Hodge\--Tate weights of \(V\) are all \(\leq n\), then
		\[
			E_+^n(V)=E_+(V).
		\]
	\end{enumerate}
\end{lemma}
\begin{proof}
	Both statements follow from Proposition~\ref{proposition:t_V}.
\end{proof}

\begin{proposition} \label{proposition:universal_extension_vector_bundles}
	If \(V\) is de Rham, then there exists an isomorphism of short exact sequences of almost \(\Cp\)\=/representations of \(G_K\)
	\[
		\begin{tikzcd}
			0 \ar{r} & V \ar{r} \ar{d}[sloped]{\sim} & E_+^n(V) \ar{r} \ar{d}[sloped]{\sim} & \Fil^{-n} \hat{t}_V(\Qpbar) \ar{r} \ar{d}[sloped]{\sim} & 0 \\
			0 \ar{r} & \HH^0(\XFF,\Ec(V)) \ar{r} & \HH^0(\XFF,\Ec_+^n(V)) \ar{r} & \HH^0(\XFF,\Fc_+^n(V)) \ar{r} & 0.
		\end{tikzcd}
	\]
\end{proposition}
\begin{proof}
	The statement is proved similarly to~\cite[Proposition~3.3.1 and Lemma~3.3.2]{Ponsinet2024:1}.
	By Lemma~\ref{lemma:HN_geq0}, the Harder\--Narasimhan slopes of the sheaves
	\[
		0 \rightarrow \Ec(V) \rightarrow \Ec_+^n(V) \rightarrow \Fc_+^n(V) \rightarrow 0
	\]
	are all \(\geq 0\).
	Thus, by Proposition~\ref{proposition:cohomology_HN}, there exists a commutative diagram of topological \(G_K\)\=/modules
	\[
		\begin{tikzcd}
			0 \ar{r} & V \ar{r} & E_\e(V) \ar{r} & (\BdR/\BdR^+)\otimes_{\Qp} V \ar{r} & 0 \\
			0 \ar{r} & \HH^0(\XFF,\Ec(V)) \ar{r} \ar{u}[sloped]{\sim} & \HH^0(\XFF,\Ec_+^n(V)) \ar{r} \ar{u} & \HH^0(\XFF,\Fc_+^n(V)) \ar{r} \ar{u} & 0.
		\end{tikzcd}
	\]
	where the bottom rows is a short exact sequence of almost \(\Cp\)\=/representations of \(G_K\) by Theorem~\ref{theorem:almost_coh}.
	Since \(V\) is de Rham, there are isomorphisms of torsion \(\BdR^+\)\=/representations
	\[
		\begin{split}
			\Fil^{-n} \hat{t}_V(\Qpbar) & \similarrightarrow \Img\left( \Fil^{-n} t_V(\BdR^+) \rightarrow (\BdR/\BdR^+)\otimes_{\Qp} V \right) \\
			& \similarrightarrow \Img\left( \BdR^+ \otimes_K (\Fil^{-n} \DDb_\dR(V)/\Fil^0 \DDb_\dR(V)) \rightarrow (\BdR/\BdR^+)\otimes_{\Qp} V \right) \\
			& \similarrightarrow \Img\left( \frac{\BdR^+ \otimes_K \Fil^{-n} \DDb_\dR(V)}{\BdR^+ \otimes_K \Fil^0 \DDb_\dR(V)} \rightarrow \frac{\BdR \otimes_{\Qp} V}{\BdR^+\otimes_{\Qp} V} \right) \\
			& \similarrightarrow \frac{\BdR^+\otimes_{\Qp} V + \BdR^+ \otimes_K \Fil^{-n} \DDb_\dR(V)}{\BdR^+\otimes_{\Qp} V} \\
			& \similarrightarrow \HH^0(\XFF,\Fc_+^n(V)),
		\end{split}
	\]
	where the first isomorphism is from Proposition~\ref{proposition:t_V} and the last one from Theorem~\ref{theorem:almost_coh}.
\end{proof}

The combination of Proposition~\ref{proposition:universal_extension_vector_bundles} and Proposition~\ref{proposition:HN_E+n} yields the following.

\begin{corollary} \label{corollary:E0}
	If \(V\) is de Rham, then there is a natural isomorphism of \(p\)\=/adic representation of \(G_K\)
	\[
		E_+^n(V)^0 \similarrightarrow V^{\leq 0,> n}.
	\]
\end{corollary}

Let \(T\) be a \(G_K\)\=/stable lattice in \(V\).
We set \(E_+(V/T) = E_+(V)/T\) and \(E_+^n(V/T) = E_+^n(V)/T\), and thus the diagram~\eqref{eq:funda_+} induces a commutative diagram of topological \(G_K\)\=/modules with exact rows
\begin{equation} \label{eq:funda_+_T}
	\begin{tikzcd}
		0 \ar{r} & V/T \ar{r} & E_+(V/T) \ar{r} & \hat{t}_V(\Qpbar) \ar{r} & 0 \\
		0 \ar{r} & V/T \ar{r} \ar[equal]{u} & E^n_+(V/T) \ar{r} \ar{u} & \Fil^{-n} \hat{t}_V(\Qpbar) \ar{r} \ar{u} & 0.
	\end{tikzcd}
\end{equation}
We set \(E_\delta(V/T) = (E_+(V/T))_\delta\) and \(E_\delta^n(V/T) = (E_+^n(V/T))_\delta\), and thus, by Lemma~\ref{lemma:delta_ses}, the diagram~\eqref{eq:funda_+_T} induces a commutative diagram of discrete \(G_K\)\=/modules with exact rows
\begin{equation} \label{eq:grp_pts}
	\begin{tikzcd}
		0 \ar{r} & V/T \ar{r} & E_\delta(V/T) \ar{r} & t_V(\Qpbar) \ar{r} & 0 \\
		0 \ar{r} & V/T \ar{r} \ar[equal]{u} & E^n_\delta(V/T) \ar{r} \ar{u} & \Fil^{-n} t_V(\Qpbar) \ar{r} \ar{u} & 0
	\end{tikzcd}
\end{equation}

\begin{remark}
	Fontaine has defined the \emph{group of points} \(E_\disc(V/T)\) associated with \(V/T\) as the image of \(E_\delta(V/T)\) in \(E_+(V/T)\).
	As in the article~\cite{Ponsinet2024:1}, we use different notation to highlight the different topologies: \(E_\delta(V/T)\) is a discrete \(G_K\)\=/module, while \(E_\disc(V/T)\) is endowed with the subspace topology from \(E_+(V/T)\).
\end{remark}

\begin{proposition} \label{proposition:BK_delta}
	Let \(L\) be an algebraic extension of \(K\).
	The commutative diagram of discrete \(G_K\)\=/modules~\eqref{eq:grp_pts} induces a commutative diagram whose rows are exact
	\[
	\begin{tikzcd}
		0 \ar{r} & \HH^1_e(L,V/T) \ar{r} & \HH^1(L,V/T) \ar{r} & \HH^1(L,E_\delta(V/T)) \ar{r} & 0 \\
		0 \ar{r} & \Fil^{-n} \HH^1_e(L,V/T) \ar{r} \ar{u} & \HH^1(L,V/T) \ar{r} \ar[equal]{u} & \HH^1(L,E_\delta^n(V/T)) \ar{r} \ar{u} & 0.
	\end{tikzcd}
	\]
\end{proposition}
\begin{proof}
	The statement for \(E_\delta(V/T)\) is \cite[Proposition~3.2.1]{Ponsinet2024:1}, while the statement for \(E_\delta^n(V/T)\) is proved similarly.
	If \(K^\prime\) is a finite extension of \(K\), then the cohomology of \(K^\prime\) of the diagram
	\[
		\begin{tikzcd}
			0 \ar{r} & V \ar{r} & E_n(V) \ar{r} & (\Fil^{-n} \BdR/\BdR^+) \otimes_{\Qp} V \ar{r} & 0 \\
			0 \ar{r} & V \ar{r} \ar[equal]{u} & E_+^n(V) \ar{r} \ar{u} & \Fil^{-n} \hat{t}_V(\Qpbar) \ar{r} \ar{u} & 0
		\end{tikzcd}
	\]
	induces the commutative diagram with exact rows
	\[
		\begin{tikzcd}
			0 \ar{r} & V^{G_{K^\prime}} \ar{r} & \Fil^{-n} \DDb_{\cris,K^\prime}(V)^{\varphi=1} \ar{r} & \Fil^{-n} t_V(K^\prime) \ar{r} & \HH^1(K^\prime,V) \\
			0 \ar{r} & V^{G_{K^\prime}} \ar{r} \ar[equal]{u} & E_+^n(V)^{G_{K^\prime}} \ar{r} \ar{u}[sloped]{\sim} & \Fil^{-n} t_V(K^\prime) \ar{r} \ar[equal]{u} & \HH^1(K^\prime,V) \ar[equal]{u}.
		\end{tikzcd}
	\]
	Hence, by definition of \(\Fil^{-n} \HH^1_e(K^\prime,V)\), there is an exact sequence
	\begin{equation} \label{eq:E_BK}
		0 \rightarrow V^{G_{K^\prime}} \rightarrow E_+^n(V)^{G_{K^\prime}} \rightarrow \Fil^{-n} t_V(K^\prime) \rightarrow \Fil^{-n} \HH^1_e(K^\prime,V) \rightarrow 0.
	\end{equation}
	The commutative diagram
	\[
		\begin{tikzcd}
			0 \ar{r} & V \ar{r} \ar{d} & E_+^n(V) \ar{r}  \ar{d} & \Fil^{-n} \hat{t}_V(\Qpbar) \ar{r}  \ar[equal]{d} & 0 \\
			0 \ar{r} & V/T \ar{r} & E_+^n(V/T) \ar{r} & \Fil^{-n} \hat{t}_V(\Qpbar) \ar{r} & 0
		\end{tikzcd}
	\]
	induces the commutative diagram with exact rows
	\begin{equation} \label{eq:E_V_T}
		\begin{tikzcd}
			0 \ar{r} & V^{G_{K^\prime}} \ar{r} \ar{d} & E_+^n(V)^{G_{K^\prime}} \ar{r} \ar{d} & \Fil^{-n} t_V(K^\prime) \ar{r} \ar[equal]{d} & \HH^1(K^\prime,V) \ar{d} \\
			0 \ar{r} & (V/T)^{G_{K^\prime}} \ar{r}  & E_+^n(V/T)^{G_{K^\prime}} \ar{r}  & \Fil^{-n} t_V(K^\prime) \ar{r} & \HH^1(K^\prime,V/T).
		\end{tikzcd}
	\end{equation}
	Hence, by definition of \(\Fil^{-n} \HH^1_e(K^\prime,V/T)\), the exact sequence~\eqref{eq:E_BK}, and the commutativity of the diagram~\eqref{eq:E_V_T}, there is an exact sequence
	\begin{equation} \label{eq:E_BK_T}
		0 \rightarrow (V/T)^{G_{K^\prime}} \rightarrow E_+^n(V/T)^{G_{K^\prime}} \rightarrow \Fil^{-n} t_V(K^\prime) \rightarrow \Fil^{-n} \HH^1_e(K^\prime,V/T) \rightarrow 0.
	\end{equation}
	Therefore, by the exact sequence~\eqref{eq:E_BK_T}, the short exact sequence of discrete \(G_K\)\=/modules
	\[
		0 \rightarrow V/T \rightarrow E_\delta^n(V/T) \rightarrow \Fil^{-n} t_V (\Qpbar) \rightarrow 0
	\]
	induces a short exact sequence
	\begin{equation} \label{eq:delta_BK}
		0 \rightarrow \Fil^{-n} \HH^1_e(K^\prime,V/T) \rightarrow \HH^1(K^\prime,V/T) \rightarrow \HH^1(K^\prime,E_\delta^n(V/T)) \rightarrow 0.
	\end{equation}
	By definition of \(\Fil^{-n} \HH^1_e(L,V/T)\) and by continuity of Galois cohomology with coefficients in discrete modules~\cite[I \S 2.2 Proposition~8]{Serre1994}, the limit of the short exact sequences~\eqref{eq:delta_BK} over the finite extensions \(K^\prime\) of \(K\) contained in \(L\) with transition morphisms the restriction maps yields
	\[
		0 \rightarrow \Fil^{-n} \HH^1_e(L,V/T) \rightarrow \HH^1(L,V/T) \rightarrow \HH^1(L,E_\delta^n(V/T)) \rightarrow 0.
	\]
\end{proof}

The combination of Proposition~\ref{proposition:BK_delta} and Lemma~\ref{lemma:E+_Fil} yields the following.

\begin{corollary} \label{corollary:BK_HT}
	Let \(L\) be an algebraic extension of \(K\).
	Assume that \(V\) is de Rham.
	\begin{enumerate}
		\item If the Hodge\--Tate weights of \(V\) are all \(\leq 0\), then
		\[
			\HH^1_e(L,V/T) = 0.
		\]
		\item If the Hodge\--Tate weights of \(V\) are all \(\leq n\), then
		\[
			\Fil^{-n} \HH^1_e(L,V/T) = \HH^1_e(L,V/T).
		\]
	\end{enumerate}
\end{corollary}

\subsection{Bloch\--Kato groups and Galois theory of \(\BdR^+\)}
\label{subsec:universal_norms_Gal_periods}

Let \(V\) be a \(p\)\=/adic representation of \(G_K\).
Let \(L\) be an algebraic extension of \(K\).

If \(V\) is de Rham and if \(n \geq 1\) is an integer, then we denote by \(T^{\leq 0,> n}\) the image of \(T\) in \(V^{\leq 0,> n}\), and the quotient map \(V/T \rightarrow V^{\leq 0,> n}/T^{\leq 0,> n}\) induces a morphism
\[
	\pi_{0,n} : \HH^1(L,V/T) \rightarrow \HH^1(L,V^{\leq 0,> n}/T^{\leq 0,> n}).
\]
Note that if the Hodge\--Tate weights of \(V\) are all \(\leq n\), then the representation \(V^{\leq 0,> n}\) is the maximal quotient representation of \(V\) whose Hodge\--Tate weights are all \(\leq 0\), and we then simply denote the representation \(V^{\leq 0,> n}\) by \(V^{\leq 0}\), the lattice \(T^{\leq 0,> n}\) by \(T^{\leq 0}\), and the map \(\pi_{0,n}\) by
\[
	\pi_0 : \HH^1(L,V/T) \rightarrow \HH^1(L,V^{\leq 0}/T^{\leq 0}).
\]

\begin{theorem} \label{theorem:main}
	Let \(n \geq 1\) be an integer.
	If \(V\) is de Rham and if \(\hat{L}\) is a perfectoid field such that \(L\) is dense in \(\Bb_n^{G_L}\), then the map \(\pi_{0,n}\) induces an isomorphism
	\[
		\HH^1(L,V/T)/\Fil^{-n} \HH^1_e(L,V/T) \similarrightarrow \HH^1(L,V^{\leq 0,> n}/T^{\leq 0,> n}).
	\]
\end{theorem}
\begin{proof}
	By Proposition~\ref{proposition:BK_delta}, there is an isomorphism
	\[
		\HH^1(L,V/T)/\Fil^{-n} \HH^1_e(L,V/T) \similarrightarrow \HH^1(L,E_\delta^n(V/T)).
	\]
	Since \(L\) is dense in \(\Bb_n^{G_L}\), the module \((\Fil^{-n} t_V(\Qpbar))^{G_L}\) is dense in \((\Fil^{-n} \hat{t}_V(\Qpbar))^{G_L}\) by Proposition~\ref{proposition:t_V}.
	Moreover, since \(\Fil^{-n} \hat{t}_V(\Qpbar)\) is an object of \(\Cc^{+\infty}(G_K)\) and thus of \(\Cc^{\neq 0}(G_K)\), the group \(\HH^1(L,\Fil^{-n} \hat{t}_V(\Qpbar))\) is trivial by Proposition~\ref{proposition:cohomology_almost_perfectoid}.
	Therefore, by Corollary~\ref{corollary:delta_perfectoid} and Corollary~\ref{corollary:E0}, there are isomorphisms
	\[
		\HH^1(L,E_\delta^n(V/T)) \similarrightarrow \HH^1(L,E_+^n(V/T)) \similarrightarrow \HH^1(L,V^{\leq 0,> n}/T^{\leq 0,> n}).
	\]
\end{proof}

\begin{corollary} \label{corollary:main}
	Let \(n \geq 1\) be an integer.
	Assume that \(V\) is de Rham and that \(\hat{L}\) is a perfectoid field such that \(L\) is dense in \(\Bb_n^{G_L}\).
	\begin{enumerate}
		\item If the quotient representation \(V^{\leq 0,> n}\) is trivial, then
		\[
			\HH^1_e(L,V/T) = \HH^1(L,V/T).
		\]
		\item If the Hodge\--Tate weights of \(V\) are all \(\leq n\), then the map \(\pi_0\) induces an isomorphism
		\[
			\HH^1(L,V/T)/\HH^1_e(L,V/T) \similarrightarrow \HH^1(L,V^{\leq 0}/T^{\leq 0}).
		\]
	\end{enumerate}
\end{corollary}
\begin{proof}
	The first statement follows immediately from Theorem~\ref{theorem:main}.
	The second statement follows from Theorem~\ref{theorem:main} and Corollary~\ref{corollary:BK_HT}.
\end{proof}

\printbibliography[heading=bibintoc]
\end{document}